\documentclass[12pt]{amsart}

\usepackage[utf8]{inputenc}
\usepackage{amsmath}
\usepackage{amsthm}
\usepackage{amsfonts}
\usepackage{amssymb}
\usepackage{tabu}
\usepackage{latexsym} 
\usepackage[margin=1in]{geometry}
\usepackage{enumerate} 
\usepackage{mathrsfs}
\usepackage[usenames,dvipsnames]{color} 
\usepackage[colorlinks=true,linkcolor=Cerulean,pagebackref,citecolor=Mahogany]{hyperref} 
\usepackage{tikz}
\usepackage[parfill]{parskip}
\usetikzlibrary{matrix,arrows,decorations.pathmorphing}
\usepackage{tikz-cd}
\usepackage[all]{xy}
\usepackage{setspace}
\usepackage{tabularx}
\usepackage{xcolor}
\usepackage{framed}
\usepackage{stmaryrd} 
\usepackage{graphicx}

\newcommand{\Ext}{\mathrm{Ext}}

\newcommand{\Hom}{\mathrm{Hom}}

\newcommand{\CM}{\mathrm{CM}}

\newcommand{\depth}{\mathrm{depth }}

\newcommand{\CC}{\mathcal{C}}

\newcommand{\ZZ}{\mathbb{Z}}

\newcommand{\spec}{\mathrm{Spec}}

\newcommand{\coker}{\mathrm{coker}}

\newtheorem{theorem}{Theorem}[section]
\newtheorem*{theorem*}{Theorem}
\newtheorem{lemma}[theorem]{Lemma}
\newtheorem*{lemma*}{Lemma}
\newtheorem{proposition}[theorem]{Proposition}
\newtheorem*{proposition*}{Proposition}
\newtheorem{corollary}[theorem]{Corollary}
\newtheorem*{corollary*}{Corollary}

\theoremstyle{definition}
\newtheorem{definition}[theorem]{Definition}
\newtheorem*{definition*}{Definition}
\newtheorem{example}[theorem]{Example}
\newtheorem{remark}[theorem]{Remark}

\newtheorem{introtheorem}{Theorem}

\numberwithin{equation}{theorem}
\def\CC{\mathbb{C}}

\def\N{\mathbb{N}}

\def\T{\mathcal{T}}

\def\xx{\boldsymbol{x}}
\def\yy{\boldsymbol{y}}

\renewcommand{\S}{\mathsf{S}}

\def\CCC{\mathcal{C}}
\def\ind{\mathsf{ind}}
\def\coind{\mathsf{coind}}

\usepackage[normalem]{ulem}

\newcommand{\mc}[1]{\ensuremath{\mathcal{#1}}}   
\newcommand{\mf}[1]{\ensuremath{\mathfrak{#1}}}   

\title{Penrose tilings, infinite friezes, and the $A_\infty$-singularity}

\author[\"{O}.~Esentepe]{\"{O}zg\"{u}r Esentepe}
\email{ozgur.esentepe@uni-graz.at}
\address{Institut für Mathematik und Wissenschaftliches Rechnen, Universität Graz,
Heinrichstraße 36, 8010 Graz, Austria}

\author[E.~Faber]{Eleonore Faber}
\email{eleonore.faber@uni-graz.at}
\address{
Institut f\"ur Mathematik und Wissenschaftliches Rechnen,
Universit\"at Graz,
Heinrichstr.~36,
A-8010 Graz, Austria and School of Mathematics, University of Leeds, LS2 9JT Leeds, UK
}

\date{\today}
\thanks{ Work of both authors was supported by EPSRC grant EP/W007509/1. This material is based upon work supported by the National Science Foundation under Grant No. DMS-1928930 and by the Alfred P. Sloan Foundation under grant G-2021-16778, while E.F.~ was in residence at the Simons Laufer Mathematical Sciences Institute (formerly MSRI) in Berkeley, California, during the Spring 2024 semester.
}

\begin{document}

\begin{abstract}
We study Penrose tilings of the plane $\mathbb{R}^2$ and nonperiodic infinite frieze patterns from the point of view of Cohen--Macaulay representation theory:  Triangulations of the completed $\infty$-gon correspond to subcategories of the Frobenius category $\mc{C}_2=\CM_{\ZZ}(\CC[x,y]/(x^2))$, the singularity category of the curve singularity of type $A_\infty$. We relate Penrose tilings to certain triangulations of the completed $\infty$-gon, and thus to the corresponding subcategories of $\mc{C}_2$.  We then extend the cluster character of Paquette and Y{\i}ld{\i}r{\i}m for a triangulated category modelling said triangulations to our setting. This allows us to define nonperiodic infinite friezes patterns coming from triangulations of the completed $\infty$-gon and in particular from Penrose tilings. 
\end{abstract}

\maketitle

\thispagestyle{empty}

\section{Introduction}

Penrose tilings are aperiodic tilings of the plane $\mathbb{R}^2$ which were investigated by the renowned physicist Roger Penrose in the 1970s \cite{Penrose74paper,Penrose78paper}. They appear everywhere from medieval Islamic geometric patterns to the study of quasicrystals which earned Shechtman the 2011 Nobel Prize in Chemistry \cite{nobel-prize}.

Frieze patterns were introduced by Coxeter in \cite{Coxeter} motivated by Gauss's \textit{pentagramma mirificum}. They are grids of positive integers satisfying a determinant rule and together with Conway, Coxeter showed that there is a one-to-one correspondence between frieze patterns and triangulated polygons \cite{CoCo1,CoCo2}. These friezes have subsequently appeared in the representation theory of finite dimensional algebras in the last two decades and they are a topic of active research. We refer to \cite{Crossroads, Baur, Pressland-LMS, FaberICRA} for surveys on frieze patterns.

By the $A_\infty$-singularity, we mean the commutative ring $S = \CC[x,y]/(x^2)$. This is a nonreduced hypersurface ring of Krull dimension one - hence it has a nonisolated singularity - and $\spec(S)$ is a plane curve singularity of type $A_\infty$. It is referred to as the $A_\infty$-singularity as it can be seen (in some sense) as a limit of $A_n$-singularities given by $\CC[x,y]/(x^2+y^{n+1})$, see \cite{Siersma}.

Our goal in this paper is to draw connections between these classical and beautiful objects using the representation theory and the combinatorics associated to the commutative ring $S$. To start with, we equip $S$ with a $\ZZ$-grading and consider the category $\mc{C}_2$ of $\ZZ$-graded maximal Cohen-Macaulay $S$-modules. In \cite{ACFGS}, the authors studied this Frobenius category as the Grassmannian category of infinite rank associated to an infinite version of the Grassmannian of planes and they classified cluster tilting subcategories of it. In particular, they showed that there exists a one-to-one correspondence between cluster tilting subcategories of $\mc{C}_2$ and fountain triangulations of the completed $\infty$-gon. We are going to recall the necessary background in Section \ref{Sec:Prelim}.
Let us note that there are several other approaches to categorify the combinatorics of $\infty$-gons, for example \cite{Holm-Jorgensen, F16, Paquette-Yildirim, IgusaTodorov-cyclic, Canakci-Kalck-Pressland}, cf.~the discussion in the introduction of \cite{ACFGS} where some of these approaches are compared.

One of the fundamental tools in the representation theory of cluster algebras and cluster categories is the notion of mutation. In the finite case, the rules of mutation are well-defined and well-understood. In particular, for cluster algebras of type $A$, mutation can be described combinatorially as flips of diagonals in a triangulated polygon. In the infinite case, there are simply too many triangulations of an $\infty$-gon. There are many pairs of triangulations which may not be connected via a series of finitely many mutations. To attack this problem, Baur and Gratz introduced in \cite{Baur-Gratz} strong mutation equivalence and transfinite mutation equivalence. They proved that any two triangulations of the completed $\infty$-gon are transfinitely mutation equivalent. This is where our motivation for our main theorem comes from: finite mutation equivalence gives us a very wild behaviour and transfinite mutation equivalence gives us only one equivalence class. Can we find a new type of mutation equivalence which lives somewhere between these two extremes? Our answer is affirmative. We introduce an equivalence relation $\sim$ on the set of cluster tilting subcategories of $\mc{C}_2$ and a special class of cluster tilting subcategories of $\mc{C}_2$ which allows us to prove our first main theorem.
\begin{introtheorem}\label{intro-theorem-penrose}
    Let $2^\mathbb{N}$ denote the set of all sequences on $\{0,1\}$ and $\mathcal{P}$ denote the space of Penrose tilings.
    \begin{enumerate}
        \item There exists a surjection 
        \begin{align*}
           \psi \colon \{\text{Cluster tilting subcategories of } \mc{C}_2\}\twoheadrightarrow 2^{\mathbb{N}} \times 2^{\mathbb{N}} \ .
        \end{align*}
        \item For two cluster tilting subcategories $\S_1, \S_2$ of $\mc{C}_2$, we have $\S_1 \sim \S_2$ if and only if $\psi(\S_1)$ and $\psi(\S_2)$ differ only in finitely many places.  
        \item There exists a one-to-one correspondence 
        \begin{align*}
        \frac{\{\text{Special cluster tilting subcategories of } \mc{C}_2\}}{\sim} \longleftrightarrow \frac{\mathcal{P} \times \mathcal{P}}{\text{isometry}} \ .
        \end{align*}
    \end{enumerate}
\end{introtheorem}
Our Section \ref{Sec:Penrose-friezes} is devoted to defining what these special cluster tilting subcategories are and to proving this theorem. In Section \ref{Sec:CC->Frieze}, we move to infinite friezes. The representation theoretic link between friezes and cluster algebras has been well-established since the beginning of 2000s, starting with the work of Caldero and Chapoton \cite{Caldero-Chapoton}. The main tool here is a function called \textit{cluster character}. However, this map often requires some finiteness conditions. Recently, \cite[Section 6]{Paquette-Yildirim}, went through the technical details and defined a cluster character $\chi_\T$ on a triangulated category associated to the $\infty$-gon where $\T$ is a cluster tilting subcategory of this triangulated category. Utilising their methods, given a cluster tilting subcategory $\S$  of the Frobenius category $\mc{C}_2$, we define a cluster character $X^\S$ on (a certain subcategory of) $\mc{C}_2$. Our second main theorem is as follows.
\begin{introtheorem}[see Theorem \ref{Thm:half-friezes-coeffs} and Definition \ref{Def:inf-frieze-coeff}]\label{intro-theorem-friezes}
    Let $\S$ be a cluster tilting subcategory of the Frobenius category $\mc{C}_2$. Then, the cluster character $X^\S$ yields an infinite fountain frieze $\mc{F}$ with coefficients.
\end{introtheorem}
We devote Section \ref{Sec:Examples} to examples, in particular, we compute a fountain frieze and a frieze coming from a leapfrog triangulation of the $\infty$-gon.

\section{Preliminaries} \label{Sec:Prelim}

In this section, we are going to introduce some notation and give preliminaries that will be needed for the rest of the paper.

\subsection{Maximal Cohen--Macaulay modules over the $A_\infty$-singularity} \label{Sub:MCM-Frobenius-cat}
Given a commutative ring $S$, a finitely generated $S$-module $M$ is \textit{maximal Cohen--Macaulay} if $\depth_S M_\mf{p} = \dim M_\mf{p} = \dim S_\mf{p}$ for every prime ideal $\mf{p}$ of $S$. When $S$ is a commutative Gorenstein ring, this condition is equivalent to requiring that $\Ext_S^i(M,S) =0$ for all $i > 0 $ and maximal Cohen--Macaulay modules coincide with Gorenstein-projective modules, see e.g. \cite[Chapter 4]{Buchweitz}. By virtue of the Depth Lemma, the category of maximal Cohen--Macaulay modules $\CM(S)$ is an exact category and in the Gorenstein case the above Ext-vanishing gives us a Frobenius structure. In this special case, the category of maximal Cohen--Macaulay modules modulo projectives, i.e. \textit{the stable category} of maximal Cohen--Macaulay modules $\underline{\CM}(S)$, has the natural structure of a triangulated category with the shift functor $\Omega^{-1}$ (the coszygy functor) and is equivalent (as triangulated category) to the singularity category of $S$ which is defined as the Verdier quotient of the bounded derived category of $S$-modules by the subcategory of perfect complexes by the celebrated work of Buchweitz \cite{Buchweitz}.

Next, we continue with another celebrated work from the 1980's: Eisenbud's matrix factorizations \cite{Eisenbud}. Over a local hypersurface ring $\CC\llbracket x_0,\ldots,x_d\rrbracket/(f)$ given by a polynomial $f$, every maximal Cohen--Macaulay module has a two-periodic minimal projective resolution. Moreover, the ranks of the free modules appearing in this resolution are all equal, say $l >0$. This gives us a pair of square matrices $A$ and $B$ with entries in $\CC\llbracket x_0,\ldots, x_d\rrbracket$ such that $AB = BA = f \mathbb{I}_l$. The pair $(A,B)$ is called a \textit{matrix factorization} of $f$. In fact, the assignment $(A,B) \mapsto \coker A$ yields an equivalence of categories between the category of matrix factorizations and maximal Cohen--Macaulay modules. This equivalence reduces to an equivalence of triangulated categories between the stable category of maximal Cohen--Macaulay modules and the homotopy category of matrix factorizations.

One can also make these definitions in the graded setting albeit with a little extra care.

As mentioned in the introduction, we are interested in maximal Cohen--Macaulay modules over the ring $S= \CC[x,y]/(x^2)$. We will consider $S$ as a graded ring with $\deg(x) = 1$ and $\deg(y) = -1$ and we will consider the category $\CCC_2=\CM_\ZZ(S)$ of graded maximal Cohen--Macaulay modules. This is a Frobenius category and it has the Krull--Remak--Schmidt property.

The ring $S$ has countable Cohen--Macaulay type. That is, there are countably many indecomposable maximal Cohen--Macaulay modules up to isomorphism, see, for instance \cite[Chapter 14]{Leuschke-Wiegand}. In the local case, the indecomposable maximal Cohen--Macaulay $S$-modules were classified by Buchweitz--Greuel--Schreyer \cite[Proposition 4.1]{Buchweitz-Greuel-Schreyer} and all these modules are gradable. So, one has the following families of indecomposable objects in the category $\CCC_2$, see \cite[Proposition 2.1]{ACFGS}:
\begin{enumerate}
    \item the shifts of projective module $S$ given by the matrix factorization 
    \begin{align*}
    \CC[x,y](-2) \xrightarrow{1} \CC[x,y](-2) \xrightarrow{x^2} \CC[x,y],
    \end{align*}
    \item the shifts of the module $\CC[y] \cong S/(x)$ given by the matrix factorization 
    \begin{align*}
    \CC[x,y](-2) \xrightarrow{x} \CC[x,y](-1) \xrightarrow{x} \CC[x,y],
    \end{align*}
    \item the shifts of the family of the ideals: $(x,y^k)$ with $k \geq 1$ given by the matrix factorization 
    \begin{align*}
    \CC[x,y](-2) \oplus \CC[x,y](k-1) \xrightarrow{B} \CC[x,y](-1) \oplus \CC[x,y](k) \xrightarrow{A} \CC[x,y] \oplus \CC[x,y](k+1) 
    \end{align*}
    where 
    \begin{align*}
    A = B = \begin{bmatrix}
        x & y^k \\ 0 & -x
    \end{bmatrix}.
    \end{align*}
    
\end{enumerate}

\subsection{The combinatorics} \label{Sub:comb-model}

We can visualise the indecomposable graded maximal Cohen--Macaulay $S$-modules as arcs in a (completed) $\infty$-gon. By an $\infty$\textit{-gon}, we mean a disc with a discrete set of marked points on the boundary indexed by the integers. By the Bolzano--Weierstrass theorem, these marked points need to have at least one accumulation point. Our assumption is that there is a unique accumulation point which we call $\infty$. By adding $\infty$ to the set of the marked points, we talk about the \textit{completed} $\infty$-gon.

By an \textit{arc} in the completed $\infty$-gon, we mean a pair $(a,b)$ in $\ZZ \cup \{\infty\}$ with $a < b$. We say that an arc is a \textit{finite arc} if $a < b< \infty$ and an \textit{infinite arc} if $a < b = \infty$. To the shifts of the regular module $S$, we associate the \textit{boundary arcs}: the projective module $S(j)$ is matched with the arc $(-j, 1-j)$ for any $j \in \ZZ$. For the ideal $(x,y^k)(j)$ with $k \geq 1$ and $j \in \ZZ$, we have the finite arc $(-j-k, 1-j)$. Finally, for the modules $\CC[y](j)$, we have the corresponding arcs $(-j,\infty)$. \\
Conversely, an arc $(a,b)$ with $a < b$ is matched with the ideal $(x,y^{b-a-1})(1-b)$, and in particular, a boundary arc $(a,a+1)$ corresponds to $S(-a)$. The infinite arc $(a, \infty)$ corresponds to the module $\CC[y](-a)$. This gives us a bijection between the arcs on the $\infty$-gon and the set of isomorphism classes of graded maximal Cohen--Macaulay $S$-modules. \\

We say that two arcs $(a,b)$ and $(c,d)$ \textit{cross} if $a<c<b<d$ or $c<a<d<b$. Equivalently, we have that $\Ext^1$ between the corresponding modules in $\CCC_2=\CM_\ZZ(S)$ is $1$-dimensional. We will often use the combinatorial terminology when discussing properties of the category $\CCC_2$. A \textit{triangulation} of the (completed) $\infty$-gon is a maximal set of noncrossing arcs. The triangulations of the $\infty$-gon were described in \cite[Lemma 3.3]{Holm-Jorgensen} and for the completed $\infty$-gon in \cite{Baur-Gratz}. 

A set of arcs $\{(n,x_i) \colon i \in \mathbb{N}\}$ is called a \textit{right fountain} if $x_i$ is an increasing sequence. We call $n$, in this case, the \textit{fountain point} of the right fountain. Left fountains are defined similarly. A \textit{fountain} is a union of a left fountain and a right fountain with the same fountain point. If we are working with the completed $\infty$-gon, we also consider the infinite arc at the fountain point. There exists a one-to-one correspondence between cluster tilting subcategories of $\CCC_2$ and fountain triangulations of the completed $\infty$-gon, i.e., triangulations of the completed $\infty$-gon containing a fountain \cite[Theorem 4.11]{ACFGS}. For a schematic picture, see Fig.~\ref{Fig:fountain}:
\begin{figure}[h!]
\begin{tikzpicture}[scale=3,cap=round,>=latex]

    \foreach \a in {  20, 40, 60, 80, 100, 120} {
  \draw (270:0.5) .. controls (270:0.5) and (270-\a:0.4) .. (270-\a:0.5);
  }

    \foreach \a in {  20, 40, 60, 80, 100, 120} {
  \draw (270:0.5) .. controls (270:0.5) and (270+\a:0.4) .. (270+\a:0.5);
  }

  \foreach \a in {270, 290,310,330,350,10, 30} {
    \node[fill=black,circle,inner sep=0.02cm] at (\a:0.5) {};
  }

  \foreach \a in { 250,230,210,190,170, 150} {
    \node[fill=black,circle,inner sep=0.02cm] at (\a:0.5) {};
  }

\draw(270:0.5) .. controls (270:0.2) and (90:0.2) .. (90:0.5);

  \draw[thick,dotted] (30:0.4) arc (30:80:0.4);
    \draw[thick,dotted] (150:0.4) arc (150:100:0.4);

  \draw (0,0) circle(0.5cm);

  \node at (90:0.65) {$\infty$};
  \draw (90:0.5) node[fill=black,circle,inner sep=0.065cm] {} circle (0.015cm);	  

\end{tikzpicture}
\caption{A fountain triangulation.} \label{Fig:fountain}
\end{figure}
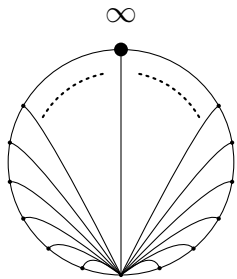  

In Fig.~\ref{Fig:Triangulations} are schematic pictures of the $5$ different types of triangulations of the completed $\infty$-gon. Note here that each triangle in the schematic triangulation may contain a triangulation by finitely many arcs. We refer to Remark \ref{Sub:triangulations} for the relations of triangulations of the (completed) $\infty$-gon to cluster tilting and maximal (almost) rigid subcategories.

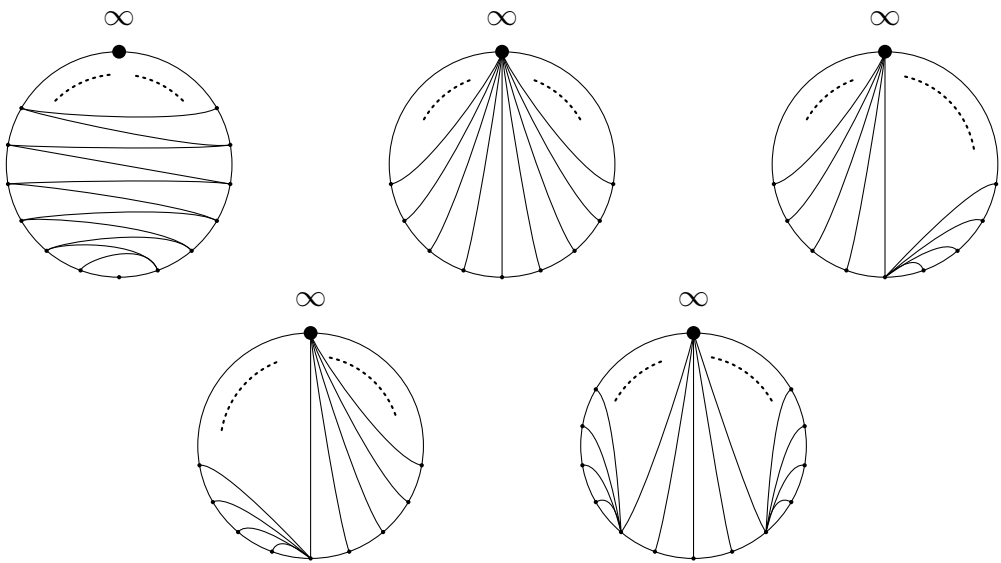
\begin{figure}[h!]
\centering
\begin{tabular}{ccc}
\begin{minipage}{0.3 \textwidth}
\begin{tikzpicture}[scale=3,cap=round,>=latex]

  \foreach \a in {   40,  60, 80, 100, 120} {
  \draw (270-\a:0.5) .. controls (270-\a:0.45) and (270+\a:0.35) .. (270+\a:0.5);
    \draw (270-\a:0.5) .. controls (270-\a:0.45) and (270+\a-20:0.35) .. (270+\a-20:0.5);
 }
    \draw (270-20:0.5) .. controls (270-20:0.45) and (270+20:0.35) .. (270+20:0.5);

  \foreach \a in {270, 290,310,330,350,10, 30} {
    \node[fill=black,circle,inner sep=0.02cm] at (\a:0.5) {};
  }

  \foreach \a in { 250,230,210,190,170,150} {
    \node[fill=black,circle,inner sep=0.02cm] at (\a:0.5) {};
  }
  

  \draw[thick,dotted] (45:0.4) arc (45:80:0.4);
    \draw[thick,dotted] (135:0.4) arc (135:95:0.4);

  \draw (0,0) circle(0.5cm);

  \node at (90:0.65) {$\infty$};
  \draw (90:0.5) node[fill=black,circle,inner sep=0.065cm] {} circle (0.015cm);	  

\end{tikzpicture}
\end{minipage}

\begin{minipage}{0.3 \textwidth}
\begin{tikzpicture}[scale=3,cap=round,>=latex]

  \foreach \a in {  280,  300,    320,   340, 20, 40,60,80} {
  \draw (90:0.5) .. controls (90:0.45) and (270+\a:0.35) .. (270+\a:0.5);
  }

  \foreach \a in {270, 290,310,330,350} {
    \node[fill=black,circle,inner sep=0.02cm] at (\a:0.5) {};
  }

  \foreach \a in { 250,230,210,190} {
    \node[fill=black,circle,inner sep=0.02cm] at (\a:0.5) {};
  }

\draw(270:0.5) .. controls (270:0.2) and (90:0.2) .. (90:0.5);

  \draw[thick,dotted] (30:0.4) arc (30:70:0.4);
    \draw[thick,dotted] (150:0.4) arc (150:110:0.4);

  \draw (0,0) circle(0.5cm);

  \node at (90:0.65) {$\infty$};
  \draw (90:0.5) node[fill=black,circle,inner sep=0.065cm] {} circle (0.015cm);	  

\end{tikzpicture}
\end{minipage}
\begin{minipage}{0.3 \textwidth}
\begin{tikzpicture}[scale=3,cap=round,>=latex]

  \foreach \a in {  280, 300,    320,   340} {
  \draw (90:0.5) .. controls (90:0.45) and (270+\a:0.35) .. (270+\a:0.5);
  }
  
    \foreach \a in {  20, 40, 60, 80} {
  \draw (270:0.5) .. controls (270:0.5) and (270+\a:0.4) .. (270+\a:0.5);
  }

  \foreach \a in {270, 290,310,330,350} {
    \node[fill=black,circle,inner sep=0.02cm] at (\a:0.5) {};
  }

  \foreach \a in { 250,230,210,190} {
    \node[fill=black,circle,inner sep=0.02cm] at (\a:0.5) {};
  }

\draw(270:0.5) .. controls (270:0.2) and (90:0.2) .. (90:0.5);

  \draw[thick,dotted] (10:0.4) arc (10:80:0.4);
    \draw[thick,dotted] (150:0.4) arc (150:110:0.4);

  \draw (0,0) circle(0.5cm);

  \node at (90:0.65) {$\infty$};
  \draw (90:0.5) node[fill=black,circle,inner sep=0.065cm] {} circle (0.015cm);	  

\end{tikzpicture}
\end{minipage}
\end{tabular}
\begin{tabular}{cc}
\begin{minipage}{0.3 \textwidth}
\begin{tikzpicture}[scale=3,cap=round,>=latex]

  \foreach \a in {  280, 300,    320,   340} {
  \draw (90:0.5) .. controls (90:0.45) and (270-\a:0.35) .. (270-\a:0.5);
  }
  
    \foreach \a in {  20, 40, 60, 80} {
  \draw (270:0.5) .. controls (270:0.5) and (270-\a:0.4) .. (270-\a:0.5);
  }

  \foreach \a in {270, 290,310,330,350} {
    \node[fill=black,circle,inner sep=0.02cm] at (\a:0.5) {};
  }

  \foreach \a in { 250,230,210,190} {
    \node[fill=black,circle,inner sep=0.02cm] at (\a:0.5) {};
  }

\draw(270:0.5) .. controls (270:0.2) and (90:0.2) .. (90:0.5);

  \draw[thick,dotted] (20:0.4) arc (20:80:0.4);
    \draw[thick,dotted] (170:0.4) arc (170:110:0.4);

  \draw (0,0) circle(0.5cm);

  \node at (90:0.65) {$\infty$};
  \draw (90:0.5) node[fill=black,circle,inner sep=0.065cm] {} circle (0.015cm);	  

\end{tikzpicture}
\end{minipage}
\begin{minipage}{0.3 \textwidth}
\begin{tikzpicture}[scale=3,cap=round,>=latex]

  \foreach \a in {      320,   340} {
  \draw (90:0.5) .. controls (90:0.45) and (270+\a:0.35) .. (270+\a:0.5);
  }
  
    \foreach \a in {  20, 40} {
  \draw (90:0.5) .. controls (90:0.5) and (270+\a:0.4) .. (270+\a:0.5);
  }


    \foreach \a in { 20, 40, 60, 80} {
  \draw (230:0.5) .. controls (230:0.5) and (230-\a:0.4) .. (230-\a:0.5);
  }

    \foreach \a in { 20, 40, 60, 80} {
  \draw (310:0.5) .. controls (310:0.5) and (310+\a:0.4) .. (310+\a:0.5);
  }

  \foreach \a in {270, 290,310,330,350,10,30} {
    \node[fill=black,circle,inner sep=0.02cm] at (\a:0.5) {};
  }

  \foreach \a in { 250,230,210,190,170,150} {
    \node[fill=black,circle,inner sep=0.02cm] at (\a:0.5) {};
  }

\draw(270:0.5) .. controls (270:0.2) and (90:0.2) .. (90:0.5);

  \draw[thick,dotted] (30:0.4) arc (30:80:0.4);
    \draw[thick,dotted] (150:0.4) arc (150:110:0.4);

  \draw (0,0) circle(0.5cm);

  \node at (90:0.65) {$\infty$};
  \draw (90:0.5) node[fill=black,circle,inner sep=0.065cm] {} circle (0.015cm);	  

\end{tikzpicture}
\end{minipage}
\end{tabular}
\caption{The five different types of triangulations of the $\infty$-gon.} \label{Fig:Triangulations}
\end{figure}

\begin{remark}\label{remark:triangulations-of-natural-numbers}
    For simplicity we will often talk about triangulations of the natural numbers in the coming sections. That is, we will be restricting ourselves to pairs $(a,b)$ with $0 \leq a < b$. We can visualise the set of all such arcs in an array where the arc $(a,b)$ belongs to the $a$th row and the $b$th column. We can then visualise a triangulation of the natural numbers by simply placing a $1$ at the $a$th row and the $b$th column if and only if the arc $(a,b)$ belongs to that triangulation. To mark the boundary, we will use 0 for the pair $(a,a)$.
\end{remark}

\begin{example}\label{running-example}
Below is part of a triangulation of the natural numbers and the corresponding array.

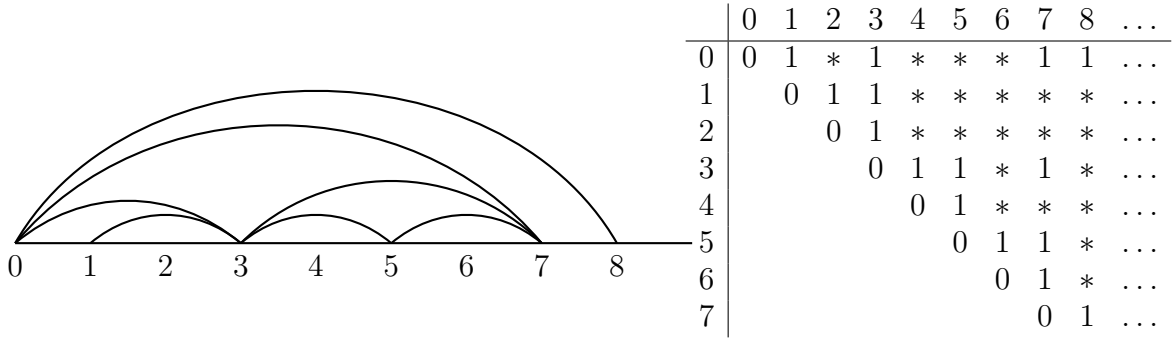
\begin{minipage}{.5\textwidth}
    \begin{tikzpicture}

        \draw[thick] (0,0) -- (9,0);
        
        \foreach \x in {0, 1, 2, 3, 4, 5, 6, 7, 8} {
            \node at (\x, -0.3) {\x};
        }
    
        \draw[thick] (0,0) to[bend left=40] (3,0);
        \draw[thick] (0,0) to[bend left=50] (7,0);
        \draw[thick] (0,0) to[bend left=60] (8,0);
        \draw[thick] (1,0) to[bend left=40] (3,0);
        \draw[thick] (3,0) to[bend left=40] (5,0);
        \draw[thick] (3,0) to[bend left=45] (7,0);
        \draw[thick] (5,0) to[bend left=40] (7,0);
    
    \end{tikzpicture}        
\end{minipage}%
\begin{minipage}{.5\textwidth}
    \begin{align*}
    \begin{array}{c|cccccccccc}
        &0&1&2&3&4&5&6&7&8&\ldots \\\hline
        0&0 & 1 & * & 1 & * & *&*& 1 & 1 & \ldots \\
        1& & 0 & 1& 1 & * & * & *& *& *& \ldots \\
        2&&& 0&1& * &* &*&*&*& \ldots \\
        3&&&&0&1& 1&*&1&*&\ldots \\
        4&&&&&0& 1&*&*&*& \ldots \\
        5&&&&&&0&1&1&*&\ldots\\
        6&&&&&&&0&1&*&\ldots \\
        7&&&&&&&&0 & 1&\ldots
    \end{array}
\end{align*}
\end{minipage}

The first row of the array shows that $(0,1), (0,3), (0,7)$ and $(0,8)$ are in the triangulation and the second row shows $(1,2)$ and $(1,3)$ are and so on.
\end{example}

Next we recall the notion of \emph{mutation} of a triangulation: Let $T$ be a triangulation of the (completed) $\infty$-gon. Then an arc $\gamma \in T$ is mutable if and only if there exists an arc $\gamma' \neq \gamma$ in the (completed) $\infty$-gon such that  
$$\mu_\gamma(T)= (T \backslash \{ \gamma \}) \cup \{ \gamma' \}$$ 
is again a triangulation of the (completed) $\infty$-gon. Note that a triangulation of a given a quadrilateral $ABCD$ is given by a choice of a diagonal. Assume that we chose the diagonal $AC$. Then, when we say we are mutating this triangulation (at the diagonal $AC$), we mean that we replace this triangulation of the quadrilateral $ABCD$ with the triangulation with the diagonal $BD$. Every finite arc in a triangulation of the (completed) $\infty$-gon (except for the boundary arcs) is contained as a diagonal in a unique quadrilateral. Hence, we can in particular talk about mutation at a given nonboundary arc in a given right fountain. For more on mutations and mutability of arcs in the $\infty$-gon, see \cite{Baur-Gratz} and \cite[Section 5]{ACFGS}.

\subsection{Penrose tilings}  
  
There are several books and survey articles introducing and discussing Penrose tilings both for a general audience and the working mathematician \cite{Grunbaum-Shephard,Gardner,Dandrea}. We will refer to the recent book \cite[Chapter 6]{Dandrea} by D'Andrea for the preliminaries that we outline below. 

Penrose tilings are parametrised by sequences on the two element set $\{0,1\}$ with the property that there are no consecutive $1$'s. More precisely we have a bijection
\begin{align*}
    \frac{\{\text{Penrose tilings}\}}{\text{isometry}} \longleftrightarrow \frac{\{\text{sequences on }\{0,1\} \text{ with no consecutive $1$'s}\}}{\text{tail equivalence}} \ .
\end{align*}
Here, tail equivalence means that we declare two sequences equivalent if and only if they agree on all but finitely many places. The set of all sequences on $\{0,1\}$, equipped with Tychonoff topology, is homeomorphic to the Cantor set. Moreover, the map that replaces a $1$ with a $10$ is a homeomorphism from the set of all sequences on $\{0,1\}$ to the subspace consisting of sequences with no consecutive 1s. Moreover, the tail equivalence is dense on this space giving that the quotient space on the right hand side of our bijection above has the trivial topology. When a quotient space has ``bad topology'', one can instead take a groupoid approach and study them using $\CC^*$-algebras and $K$-theory. This is what Connes did with his noncommutative geometry and attached a noncommutative algebraic object to the space of Penrose tilings \cite{Connes}.

In the decades following Connes' influential book on noncommutative geometry, there have also been different approaches to establish a noncommutative version of algebraic geometry. One approach is to use Serre's equivalence \cite{Serre} between modules and quasicoherent sheaves. Indeed, one can reconstruct a scheme (up to isomorphism) from the abelian category of quasicoherent sheaves on it \cite{Gabriel,Rosenberg,Rosenberg-2}. To this end, given a $\ZZ$-graded noncommutative ring $A$, one considers the \textit{noncommutative space} whose category of quasi-coherent sheaves is given by the quotient $\mathrm{Qgr} A := \mathrm{Mod}^\ZZ A / \mathrm{Fdim} A$ of the category of $\ZZ$-graded left $A$-modules by the Serre subcategory of modules that are unions of their finite dimensional submodules. By definition, the objects of $\mathrm{Qgr} A$ is the same as objects of $\mathrm{Mod}^\ZZ A $ but there are more isomorphisms in this category: in particular, two modules $M,N$ are isomorphic when $M_n \cong N_n$ for all but finitely many $n$.

Using this philosophy, in \cite{Smith} Smith showed that the space of Penrose tilings can be thought of as a noncommutative algebraic space. Precisely, he established a categorical equivalence between $\mathrm{Qgr}B$ and the category of modules over Connes' ring where $B$ is the graded noncommutative algebra $\CC\langle x, y \rangle / (x^2)$ with $x$ and $y$ both living in degree 1. Moreover, he gave a correspondence between \textit{point modules} over $B$ and Penrose tilings using the sequences on $\{0,1\}$. This was a motivating factor for us given that Smith's ring is a noncommutative version of the $A_\infty$-singularity.

\subsection{Friezes}

Here we introduce friezes, for more detailed surveys see e.g. \cite{Baur, FaberICRA, Crossroads, Pressland-LMS}.  A \textit{Conway--Coxeter frieze pattern} is a grid of numbers consisting of finitely many rows as follows: the top and the bottom rows are infinite rows of zeroes and the second to top and second to bottom rows are infinite rows of ones. All the other entries $a_{ij}$ are positive integers satisfying the \textit{frieze rule} (or \textit{unimodular rule}) which we explain below. A frieze looks like this:
\begin{align} \label{Eq:CCfrieze}
    \xymatrix@=0.3em{
    \ldots && 0 && 0 && 0 && 0 &&   \ldots \\
    \ldots & 1 && 1 && 1 && 1 && 1 & \ldots \\
    \ldots && a_{0,2} && a_{1,3} && a_{2,4} && a_{3,5} && \ldots \\
    \ldots & a_{-1,2} && a_{0,3} && a_{1,4} && a_{2,5} && a_{3,6} & \ldots \\
    \ldots && \ddots && \ddots && \ddots && \ddots &&   \ldots \\
    \ldots & a_{-2,w-1} && a_{-1,w}&& a_{0,w+1} && a_{1,w+2} && a_{2,w+3} & \ldots \\
    \ldots && 1 && 1 && 1 && 1 &&   \ldots \\
    \ldots & 0 && 0 && 0 && 0 && 0 & \ldots
    }
\end{align}
and it satisfies the frieze rule $ad - bc = 1$ for every diamond
\begin{align*}
    \xymatrix@=0.3em{
        &b&\\
        a&&d\\
        &c&
    } 
\end{align*}
The number $w$ of nontrivial rows is called the \emph{width} of the frieze. The first nontrivial row (the row of $a_{i,j}$ with $j-i = 2$ above) is called the \textit{quiddity row} and it is periodic of period $w+3$ in the horizontal direction. 

To reconcile friezes with our notation for arcs in the $\infty$-gon, by rotation/reflection, we can write the above frieze as follows.
\begin{align} \label{Eq:array-frieze}
    \begin{array}{c|cccccccccc}
        &\ddots &\ddots&\ddots&\ddots&\ddots&\ddots&\ddots&&& \\
        0&0 & 1 & a_{0,2} & a_{0,3} & \ldots & a_{0,w+1} & 1 & 0 &  &\\
        1& & 0 & 1& a_{1,3} & a_{1,4} & \ldots & a_{1,w+2} & 1& 0 &  \\
        2&&& 0&1& a_{2,4} &a_{2,5} &\ldots&a_{2,w+3}&1& 0 \\
        3&&&&0&1& a_{3,5}&a_{3,6}&\ldots&a_{3,w+4}&1 \\
        4&&&&&0& 1&\ddots&\ddots&\ddots& \ddots
    \end{array}
\end{align}
This shape looks like the arrays corresponding to our triangulations introduced in Section \ref{Sub:comb-model}. Indeed, there exists a 1-1 correspondence between finite friezes and triangulations of finite polygons: 
\begin{theorem}[Conway--Coxeter \cite{CoCo1, CoCo2}]
	\label{Thm:CoCo}
There is a bijection between triangulated polygons with $m$ vertices and Conway--Coxeter frieze patterns of width $m-3$: 
\begin{enumerate}
\item Let $\mathcal{F}$ be a Conway--Coxeter frieze of width $m-3$ with the entries $a_{ij}$ as in \eqref{Eq:CCfrieze}. Then the entries $a_{i-1,i+1}$ for $ i \in \{ 1, \ldots, m \} $
of $\mathcal{F}$ give a triangulation of an $m$-gon $P$ with vertices $1, \ldots m$: the integer $a_{i-1,i+1}$ denotes the number of triangles of the triangulation incident to vertex $i$.
\item Conversely, let $P$ be a triangulated $m$-gon. Define the sequence $\{a_i\}_{i=1}^m$ with $a_i$ as the number of triangles incident  to the vertex $i$. Then the frieze of width $m-3$ corresponding to $P$ is the one with quiddity sequence $ \{ a_i \}_{i=1}^m $.
The remaining entries in the frieze can be calculated using the frieze rule and the horizontal periodicity.
\end{enumerate}
\end{theorem}

Further note that the entries $a_{ij}$ in the frieze correspond to the arcs $(i,j)$ in the polygon. For a given triangulation of a polygon, one may plug in $a_{ij}=1$ for each arc $(i,j)$ in the triangulation, as we did in Example \ref{running-example} when we introduced the arrays. From this data and the frieze rule one can compute the remaining entries. Thus putting $1$'s for each arc (diagonal) in the triangulation, determines the frieze.

\section{Penrose tilings and friezes from triangulations of the $\infty$-gon}  \label{Sec:Penrose-friezes}

    In this section, we will study the relation between Penrose tilings and the Frobenius category $\mc{C}_2$ of maximal Cohen-Macaulay modules over the $A_\infty$-singularity. We will also see how one can go from a frieze to a Penrose tiling. Following Remark \ref{remark:triangulations-of-natural-numbers}, we will start with looking at how to get a Penrose tiling from a triangulation of natural numbers.

\subsection{Penrose tilings from right fountains } \label{Sub:Penrose-triang}
Let $R$ be a  triangulation with a right fountain of the $\infty$-gon. We call such a triangulation $R$  a \emph{right fountain triangulation}. 
Assume that it has fountain point $0$ so that we restrict ourselves to triangulations of natural numbers. Associated to $R$, there are two natural special types of arcs in $R$. 
\begin{definition}
    An arc $(a,b)$ in $R$ is a \textit{fountain arc} if $a = 0$. We denote the set of fountain arcs of $R$ by $A_1(R)$.
\end{definition}
Note that, by definition, the set of fountain arcs of $R$ is (countably) infinite. Hence, it naturally defines a sequence $\xx^R$ on the two element set $\{0,1\}$ as follows:

\begin{align*}
    \xx^R_n = \begin{cases} 1 \quad \text{if } (0,n+1) \in R \\ 0 \quad \text{else.} \end{cases}
\end{align*}
Note that the sequence starts by checking whether $(0,2)$ is in $R$ as $(0,1)$ is contained in any triangulation. The second sequence $\yy^R$ given by the set of fountain arcs of $R$ is a strictly increasing sequence of positive integers and it is defined by the rule
\begin{align*}
    \yy^R_n = \text{the } (n+1)\text{st positive integer } \ell \text{ such that } (0,\ell) \in R.
\end{align*}

\begin{remark}
    A fountain arc $(0,n)$ corresponds to the graded module $(x, y^{n-1})(1-n)$ in $\mc{C}_2$ under our identification. Hence, the smallest power of $y$ which lies in this module lives in degree zero. In fact, the converse is also true. If the smallest power of $y$ which lies in a module of the form $(x,y^k)(j)$ lives in degree zero, then the corresponding arc is a fountain arc. Note that this works only when we take 0 to be our fountain point.
\end{remark}

\begin{example}\label{x-y-sequence-example}
    Let $R$ be as in Example \ref{running-example}. Then, we have 
    \begin{align*}
        \xx^R &= 0,1,0,0,0,1,1,\ldots \\
        \yy^R &= 3,7,8 ,\ldots \ .
    \end{align*}
Note that the sequence $\xx^R$ is given by the first row of the array in Example \ref{running-example} starting from the second column and the sequence $\yy^R$ records the column numbers whose first row contains a $1$.
\end{example}
\begin{example}
    One can construct $\xx^R$ from $\yy^R$ and vice versa. For instance, if
    \begin{align*}
        \xx^R = 1,0,0,0,0,1, 0, 0, 1,0,1,0, 0 \ldots
    \end{align*}
    we see that $(0,2), (0,7), (0,10), (0,12) \in R$ and consequently, we get 
    \begin{align*}
        \yy^R = 2, 7,10,12, \ldots
    \end{align*}
\end{example}
We will define the second special class of arcs in $R$ with the help of the next lemma.
\begin{lemma}\label{second-type-arcs}
    For any $n \geq 1$, the right fountain triangulation $R$ contains the arc $(\yy^R_n, \yy^R_{n+1})$. 
\end{lemma}
\begin{proof}
    We know that the arc $(0,\yy^R_n)$ is in $R$, by definition. Thus, it has to be contained in a triangle. Since $0$ is a vertex of this triangle, there has to be another arc $(0,v)$ that is contained in this triangle. By our definition of $\yy^R$, there are two possibilities: $v = \yy^R_{n-1}$ (if $n > 2$) and $v = (\yy^R)_{n+1}$. The third edge of these two possibilities are $(\yy^R_{n-1}, \yy^R_{n})$ (if $n > 2$) and $(\yy^R_n, \yy^R_{n+1})$. Both of these triangles belong in $R$. 
\end{proof}
\begin{remark} \label{Rmk:y0}
    The assertion of Lemma \ref{second-type-arcs} can be extended for $\yy^R_0=1$.
\end{remark}
In other words, Lemma \ref{second-type-arcs} states that if $(0,a)$ and $(0,b)$ are fountain arcs in $R$ such that there does not exist any other fountain arcs $(0,x)$ with $a < x < b$, then the arc $(a,b)$ must belong to $R$. This allows us to make the next definition.
\begin{definition}
    An arc $(a,b)$ in $R$ is said to belong to $A_2(R)$ if for some $n \geq 1$ we have $a = \yy^R_n$ and $b = \yy^R_{n+1}$.
\end{definition}

\begin{remark}\label{polygon-remark}
We remark that a right fountain triangulation $R$ contains the polygons $P_n$ with finitely many vertices $(\yy^R_n,\yy^R_n + 1, \ldots, \yy^R_{n+1} - 1, \yy^R_{n+1})$ and a triangulation of these polygons.
\end{remark}

Recall that mutation was defined in Section \ref{Sub:comb-model}. We start with the effect of a mutation at a fountain arc on the sequence $\xx$.
\begin{lemma}
    Let $m$ be a positive integer and $\mu$ be the mutation of the right fountain triangulation $R$ at the fountain arc $(0,\yy^R_m)$. Then, we have
    \begin{align*}
        \xx^{\mu(R)}_n = \begin{cases}
            \xx^R_n \quad  & \text{ if } n \neq \yy^R_m \\ 
            0 \quad & \text{ if } n = \yy^R_m.
        \end{cases}
    \end{align*}
    In particular, the two sequences $\xx^R$ and $\xx^{\mu(R)}$ differ in exactly one place.
\end{lemma}
\begin{proof}
    The mutation $\mu$ takes place on the quadrilateral with vertices $0, y_{m-1}, y_m$ and  $y_{m+1}$. It replaces the arc $(0,\yy_m^R)$ with the arc $( \yy_{m-1}^R, \yy_{m+1}^R)$ (we can put $\yy^R_0 = 1$ to deal with the $m=1$ case, cf.~Remark \ref{Rmk:y0}).

    By definition, we have $\xx^R_{\yy^R_m} = 1$. After mutation, we see that $\xx^{\mu(R)}_{\yy^R_m} = 0$. It does not change any of the other arcs and hence does not change any other spots in the sequence.
\end{proof}
\begin{example}
    Let $R$ be as in Example \ref{running-example} and let us mutate at $(0,3)$.

    \begin{minipage}{.5\textwidth}
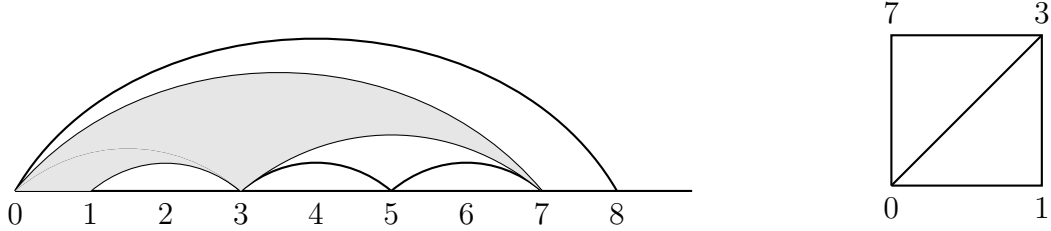
        \begin{tikzpicture}
    
            \draw[thick] (0,0) -- (9,0);
            
            \foreach \x in {0, 1, 2, 3, 4, 5, 6, 7, 8} {
                \node at (\x, -0.3) {\x};
            }
        
            \draw[thick] (0,0) to[bend left=40] (3,0);
            \draw[thick] (0,0) to[bend left=50] (7,0);
            \draw[thick] (0,0) to[bend left=60] (8,0);
            \draw[thick] (1,0) to[bend left=40] (3,0);
            \draw[thick] (3,0) to[bend left=40] (5,0);
            \draw[thick] (3,0) to[bend left=40] (7,0);
            \draw[thick] (5,0) to[bend left=40] (7,0);
            
\fill[gray!20] (1,0) to[bend left=40] (3,0)
               -- (3,0) to[bend right=40] (0,0)
               -- cycle;
\fill[gray!20] (0,0) to[bend left=40] (3,0)
to[bend left=40] (7,0)
-- (7,0) to[bend right=50] (0,0)
-- cycle;

        \end{tikzpicture}        
    \end{minipage}%
    \begin{minipage}{.2\textwidth}
        \phantom{x}
    \end{minipage}
    \begin{minipage}{.3\textwidth}

        \begin{tikzpicture}
        
            \draw[thick] (0,0) -- (2,0) -- (2,2) -- (0,2) -- (0,0);
            \node at (0, -0.3) {0};
            \node at (2, -0.3) {1};
            \node at (2, 2.3) {3};
            \node at (0, 2.3) {7};

            \draw[thick] (0,0) -- (2,2);
        
        \end{tikzpicture}

    \end{minipage}
    
    The shaded area on the left is the quadrilateral on the right with the diagonal $(0,3)$. By flipping the diagonal, we get

    \begin{minipage}{.3\textwidth}
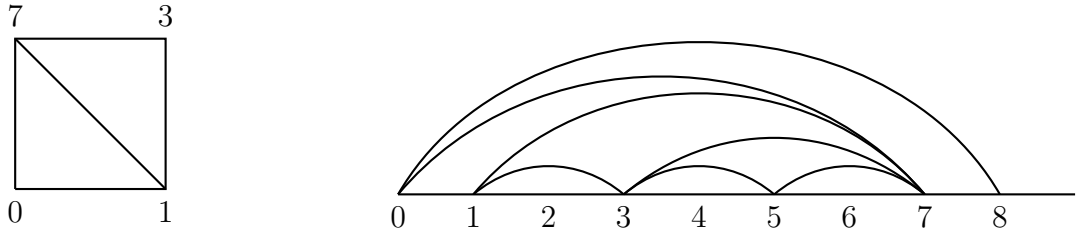
        \begin{tikzpicture}
        
            \draw[thick] (0,0) -- (2,0) -- (2,2) -- (0,2) -- (0,0);
            \node at (0, -0.3) {0};
            \node at (2, -0.3) {1};
            \node at (2, 2.3) {3};
            \node at (0, 2.3) {7};

            \draw[thick] (0,2) -- (2,0);
        
        \end{tikzpicture}
        
    \end{minipage}
    \begin{minipage}{.5\textwidth}
        \begin{tikzpicture}
    
            \draw[thick] (0,0) -- (9,0);
            
            \foreach \x in {0, 1, 2, 3, 4, 5, 6, 7, 8} {
                \node at (\x, -0.3) {\x};
            }
        
            \draw[thick] (1,0) to[bend left=50] (7,0);
            \draw[thick] (0,0) to[bend left=50] (7,0);
            \draw[thick] (0,0) to[bend left=60] (8,0);
            \draw[thick] (1,0) to[bend left=40] (3,0);
            \draw[thick] (3,0) to[bend left=40] (5,0);
            \draw[thick] (3,0) to[bend left=40] (7,0);
            \draw[thick] (5,0) to[bend left=40] (7,0);
            
        \end{tikzpicture}        
    \end{minipage}%

    and we have $        \xx^R = 0,1,0,0,0,1,1,\ldots  \rightsquigarrow \xx^{\mu(R)} = 0,0,0,0,0,1,1 \ldots$

\end{example}

Next, we see how mutating at an arc in $A_2(R)$ changes the sequence $\xx$.
\begin{lemma}
    Let $m$ be a positive integer and $\mu$ be the mutation of $R$ at the arc $(\yy^R_m, \yy^R_{m+1})$. Then, the two sequences $\xx^R$ and $\xx^{\mu(R)}$ differ in exactly one place.
\end{lemma}
\begin{proof}
    We consider the polygon $P_m$ with vertices $\yy^R_m, \yy^R_m + 1, \ldots, \yy^R_{m+1} - 1 , \yy^R_{m+1}$. Since the arc $(\yy^R_m, \yy^R_{m+1})$ belongs to this polygon, it is contained in a triangle from the triangulation of $P_m$ that is induced by $R$. Say we have the triangle ($\yy^R_m, \ell, \yy^R_{m+1}$) with $\yy^R_m < \ell < \yy^R_{m+1}$ in this triangulation. Then, locally the action of $\mu$ is to replace the arc $(\yy^R_m, \yy^R_{m+1})$ with the arc $(0,\ell)$ in the quadrilateral with vertices $0,\yy^R_m,\ell$ and $\yy^R_{m+1}$. Hence, we have
    \begin{align*}
        \xx^R_\ell = 0 \quad \text{and} \quad \xx^{\mu(R)}_\ell = 1
    \end{align*}
    and the sequence remains unchanged in other places.
\end{proof}
\begin{example}
    Once again,  let $R$ be the right fountain triangulation in Example \ref{running-example} and this time let us mutate at $(3,7)$. The number $\ell$ in the proof is 5.

    \begin{minipage}{.5\textwidth}
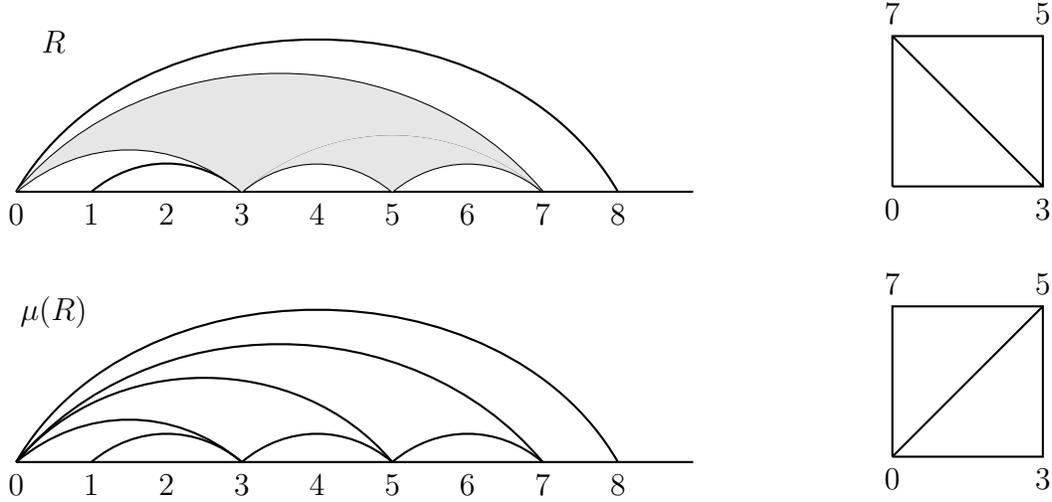
        \begin{tikzpicture}
    
            \draw[thick] (0,0) -- (9,0);
            
            \foreach \x in {0, 1, 2, 3, 4, 5, 6, 7, 8} {
                \node at (\x, -0.3) {\x};
            }
        
            \node at (0.5, 2) {$R$};

            \draw[thick] (0,0) to[bend left=40] (3,0);
            \draw[thick] (0,0) to[bend left=50] (7,0);
            \draw[thick] (0,0) to[bend left=60] (8,0);
            \draw[thick] (1,0) to[bend left=40] (3,0);
            \draw[thick] (3,0) to[bend left=40] (5,0);
            \draw[thick] (3,0) to[bend left=40] (7,0);
            \draw[thick] (5,0) to[bend left=40] (7,0);
            
\fill[gray!20] (3,0) to[bend left=40] (5,0) -- (5,0) to[bend left=40] (7,0)
               -- (7,0) to[bend right=40] (3,0)
               -- cycle;
\fill[gray!20] (0,0) to[bend left=40] (3,0)
to[bend left=40] (7,0)
-- (7,0) to[bend right=50] (0,0)
-- cycle;

        \end{tikzpicture}        
    \end{minipage}%
    \begin{minipage}{.2\textwidth}
        \phantom{x}
    \end{minipage}
    \begin{minipage}{.3\textwidth}

        \begin{tikzpicture}
        
            \draw[thick] (0,0) -- (2,0) -- (2,2) -- (0,2) -- (0,0);
            \node at (0, -0.3) {0};
            \node at (2, -0.3) {3};
            \node at (2, 2.3) {5};
            \node at (0, 2.3) {7};

            \draw[thick] (0,2) -- (2,0);
        
        \end{tikzpicture}

    \end{minipage}

    \begin{minipage}{.5\textwidth}
        \begin{tikzpicture}
    
            \draw[thick] (0,0) -- (9,0);
            
            \foreach \x in {0, 1, 2, 3, 4, 5, 6, 7, 8} {
                \node at (\x, -0.3) {\x};
            }
        
            \node at (0.5, 2) {$\mu(R)$};

            \draw[thick] (0,0) to[bend left=40] (3,0);
            \draw[thick] (0,0) to[bend left=50] (7,0);
            \draw[thick] (0,0) to[bend left=60] (8,0);
            \draw[thick] (1,0) to[bend left=40] (3,0);
            \draw[thick] (3,0) to[bend left=40] (5,0);
            \draw[thick] (0,0) to[bend left=50] (5,0);
            \draw[thick] (5,0) to[bend left=40] (7,0);

        \end{tikzpicture}        
    \end{minipage}%
    \begin{minipage}{.2\textwidth}
        \phantom{x}
    \end{minipage}
    \begin{minipage}{.3\textwidth}

        \begin{tikzpicture}
        
            \draw[thick] (0,0) -- (2,0) -- (2,2) -- (0,2) -- (0,0);
            \node at (0, -0.3) {0};
            \node at (2, -0.3) {3};
            \node at (2, 2.3) {5};
            \node at (0, 2.3) {7};

            \draw[thick] (0,0) -- (2,2);
        
        \end{tikzpicture}

    \end{minipage}
\end{example}

These are all the mutations that change the sequence $\xx$ as we see in the following lemma.
\begin{lemma}
    Let $\mu$ be a mutation of $R$ at an arc which is not of the form $(0,\yy^R_n)$ or $(\yy^R_n, \yy^R_{n+1})$. Then, we have $\xx^R = \xx^{\mu(R)}$. 
\end{lemma}
\begin{proof}
    Any arc that is not of the form $(0,\yy^R_n)$ or $(\yy^R_n, \yy^R_{n+1})$ must appear as an arc of one of the polygons $P_n$ in \ref{polygon-remark}. Then, locally the action of $\mu$ is only a mutation of $P_n$ without changing any of its boundaries. Therefore, the sequence $\xx^R$ remains unchanged.
\end{proof}
By virtue of these three lemmas, we propose the following equivalence relation.
\begin{definition}
We say that two right fountain triangulations $R_1$ and $R_2$ are equivalent and write $R_1 \sim R_2$ if $R_1$ can be mutated into $R_2$ by using only finitely many mutations at arcs of the form $(0,\yy^R_n)$ or $(\yy^R_n, \yy^R_{n+1})$ and possibly infinitely many mutations at other arcs.
\end{definition}

This definition gives us the language to state the following theorem which follows from the preceding lemmas.
\begin{theorem} \label{Thm:fountains-equiv}
    Two right fountain triangulations $R_1$ and $R_2$ are equivalent if and only if the sequences $\xx^{R_1}$ and $\xx^{R_2}$ differ in finitely many places.
\end{theorem}

The following lemma follows from the definitions of our sequences.
\begin{lemma}
    Given a right fountain triangulation $R$. For any $n \geq 1$, the following are equivalent.
    \begin{enumerate}
        \item We have $\xx^R_n = \xx^R_{n+1} = 1$.
        \item We have $\yy^R_{n+1} - \yy^R_n = 1$.
        \item The arc $(\yy^R_n, \yy^R_{n+1})$ is a boundary arc.
    \end{enumerate}
\end{lemma}
\begin{definition}\label{definition-of-special-right-fountain}
    We say that a right fountain triangulation $R$ is \textit{special} if there exists a right fountain triangulation $R'$ such that $R \sim R'$ and $\yy^{R'}_{n+1} - \yy^{R'}_n \geq 2$ for all $n \geq 1$. In other words, there does not exist any boundary arcs in $A_2(R')$. 
\end{definition}

The following proposition is just another characterisation of special right fountain triangulations.
\begin{proposition}
    A right fountain triangulation $R$ is special if and only if it contains only finitely many boundary arcs in $A_2(R)$.
\end{proposition}
Combining all this discussion, we have the following corollary.
\begin{corollary}\label{cor:special-right}
    The map $\xx$ that sends a special right fountain $R$ to the sequence $\xx^R$ induces a bijection
    \begin{align*}
\{ \emph{special right fountain triangulations} \}/\sim \quad&  \longleftrightarrow  \frac{\{\emph{sequences on }\{0,1\} \emph{ with no consecutive 1's}\}}{\emph{tail equivalence}} \\
 \textrm{which, in turn, gives us a bijection } &  \\
\{ \emph{special right fountain triangulations} \} / \sim \quad &  \longleftrightarrow  \frac{\{\textup{Penrose tilings}\}}{\textup{isometry}} \ .
    \end{align*}
\end{corollary}

    \subsection{Proof of Theorem \ref{intro-theorem-penrose}}
    
    We know from Section \ref{Sub:comb-model}, that cluster tilting subcategories of the Frobenius category $\mc{C}_2$ are in bijection with fountain triangulations of the completed $\infty$-gon. On the other hand, we have established in Section \ref{Sub:Penrose-triang} a relation between right fountain triangulations and 01-sequences. Hence, now we have all the ingredients to understand the relation between Penrose tilings and cluster tilting subcategories of $\mc{C}_2$.
    
    A fountain triangulation consists of a left fountain and a right fountain at a point $a \in \ZZ$. We might assume, up to shifts, that the fountain point is $a=0$. As all the results of Section \ref{Sub:Penrose-triang} hold true also for left fountains, given a cluster tilting subcategory $\S$ of $\mc{C}_2$ (or an appropriate shift so that the fountain point is 0), we obtain a pair of sequences $(\xx^{\S, <0}, \xx^{\S, >0})$ where the first component is given by the left fountain associated to $\S$ and the second by the corresponding right fountain. It is clear that the assignment 
    \begin{align*}
    \S \mapsto (\xx^{\S, <0}, \xx^{\S, >0})
    \end{align*}
    is surjective. This proves the first part of Theorem \ref{intro-theorem-penrose} from the Introduction. The second part of Theorem \ref{intro-theorem-penrose} follows from Theorem \ref{Thm:fountains-equiv} and the third by Corollary \ref{cor:special-right} combined with the discussion below.

    We conclude the proof by translating the language of special right/left fountains to the categorical setting. Recall that $R$ is said to be a special right fountain triangulation if it is equivalent to a right fountain triangulation $R'$ so that $A_2(R')$ does not contain a boundary arc. distinguished  classes $A_1(R)$ and $A_2(R)$ of arcs in $R$ correspond to distinguished full subcategories of a cluster tilting subcategory $\S$. Moreover, the boundary arcs correspond to projective objects in $\mc{C}_2$. Hence, with all the identifications we have made so far, for a given cluster tilting subcategory $\S$ up to shift, it is special if it is equivalent to a cluster tilting subcategory $\S'$ such that the corresponding special subcategory $A_2(\S')$ is \textit{stable} in the sense that it does not contain any projective objects. \hfill \qed

\subsection{Friezes and Penrose Tilings} \label{Sec:FriezesPenrose}

Using the bijection from Section \ref{Sub:Penrose-triang}, we relate Penrose tilings to certain arrays, which we call \emph{half-friezes}: following Conway and Coxeter, we want to associate a quiddity sequence to a fountain triangulation of the $\infty$-gon and construct a frieze from that sequence. For simplicity assume that $0$ is the fountain point. Here we quickly run into problems.  There are infinitely many arcs starting at $0$, which would give us as quiddity $a_0=\infty$. Consequently, any arc crossing $0$, that is, any arc of the form $(-i,j)$, with $i,j >0$, would have a non-integer value. However, all other elements of the quiddity sequence and the remaining arcs are well-defined integers. Therefore we make the following

\begin{definition} \label{Def:infinite-half-frieze} 
    An \emph{infinite right half-frieze} at $r \in \ZZ$ is a set of elements $m_{a,b}$ in a commutative ring $R$ with unit $1$,  with $a,b \in \ZZ$ such that $r \leq a \leq b$, 
    which satisfies the following rules:
    \begin{enumerate}[(1)]
        \item $m_{a,a} = 0$ for any $a \geq r$. 
        \item $m_{a,a+1} = 1$ for any $a \geq r$. 
        \item $m_{a,b} \neq 0$ for any pair $(a,b)$ with $r \leq a < b$.
        \item $m_{a,b} \; m_{a+1, b+1} - m_{a+1, b} \;m_{a, b+1} = 1$ for any pair $a,b$ with $r \leq a < b$. 
    \end{enumerate}
    Similarly, one can define an \emph{infinite left half-frieze} as the set of non-zero elements $m_{a,b}$ with constant $l \in \ZZ$, where the following hold:
    \begin{enumerate}[(1')]
            \item $m_{a,a} = 0$ for any $a \leq l$. 
                    \item $m_{a-1,a} = 1$ for any $a \leq l$. 
        \item $m_{a,b} \neq 0$ for any pair $(a,b)$ with $a < b \leq l$.
        \item $m_{a-1,b-1} \; m_{a, b} - m_{a-1, b} \;m_{a, b-1} = 1$ for any pair $a,b$ with $a < b \leq l$. 
            \end{enumerate}
We define an \emph{infinite half-frieze} or \emph{fountain frieze} at $r=l$ to be the union of the left half frieze at $r$ and the right half frieze at $r$.

 In the following we will consider the entries $m_{a,b} \in \ZZ_{>0}$ with $b >a$, that is, we will only look at \emph{integral} (right/left) half-friezes.
    When it is clear from the context, we will just use the term frieze without the adjectives.
The sequence $m_{a,a+2}$ with $a \neq r-1$ ($a \geq r$ resp.~$a+2 \leq l$) is called the \textit{quiddity sequence} of the (right/left) half-frieze. 
\end{definition}
It is clear from our rules (and it is easy to see when visualised on an array where $m_{a,b}$ lives in the $a$th row and $b$th column) that a quiddity sequence determines the frieze. We refer to \cite{Baur-Parsons-Tschabold} for more on infinite friezes and their quiddity sequences. 

\begin{definition} \label{def:frieze-from-fountain}
Given a right fountain triangulation $R$ of the completed $\infty$-gon (with fountain point $f \in \ZZ$), we can define a \emph{right half-frieze from the fountain triangulation $R$} by  the quiddity sequence 
\begin{align}\label{our-quiddity-rule}
m_{a,a+2} = \text{the number of triangles in } R \text{ that the vertex } a+1 \text{ is adjacent to}.
\end{align}
\end{definition}
\begin{example}\label{frieze-triangulation-example}
    Let us see this on our running example $R$.

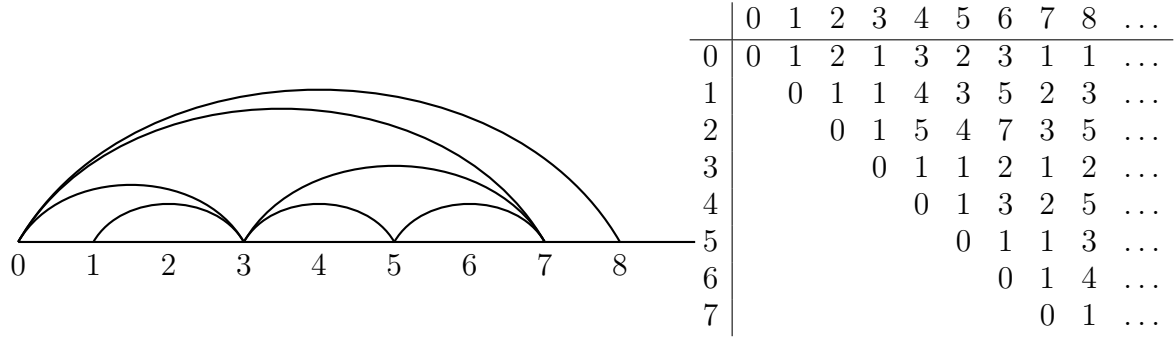
    \begin{minipage}{.5\textwidth}
        \begin{tikzpicture}
    
            \draw[thick] (0,0) -- (9,0);
            
            \foreach \x in {0, 1, 2, 3, 4, 5, 6, 7, 8} {
                \node at (\x, -0.3) {\x};
            }
        
            \draw[thick] (0,0) to[bend left=60] (3,0);
            \draw[thick] (0,0) to[bend left=60] (7,0);
            \draw[thick] (0,0) to[bend left=60] (8,0);
            \draw[thick] (1,0) to[bend left=60] (3,0);
            \draw[thick] (3,0) to[bend left=60] (5,0);
            \draw[thick] (3,0) to[bend left=60] (7,0);
            \draw[thick] (5,0) to[bend left=60] (7,0);
        
        \end{tikzpicture}        
    \end{minipage}%
    \begin{minipage}{.5\textwidth}
        \begin{align*}
        \begin{array}{c|cccccccccc}
            &0&1&2&3&4&5&6&7&8&\ldots \\\hline
            0&0 & 1 & 2 & 1 & 3 & 2&3& 1 & 1 & \ldots \\
            1& & 0 & 1& 1 & 4 & 3 & 5& 2& 3& \ldots \\
            2&&& 0&1& 5 &4 &7&3&5& \ldots \\
            3&&&&0&1& 1&2&1&2&\ldots \\
            4&&&&&0& 1&3&2&5& \ldots \\
            5&&&&&&0&1&1&3&\ldots\\
            6&&&&&&&0&1&4&\ldots \\
            7&&&&&&&&0 & 1&\ldots
        \end{array}
    \end{align*}
    \end{minipage}

    Note that for each $\yy_n^R$, if we stop the array at the column $i$, we get a finite frieze thanks to the finite polygon with vertices $0, 1, \ldots, \yy^R_n$. For instance, since our $\yy_1^R = 3$, we have the finite frieze

    \begin{minipage}{.2\textwidth}
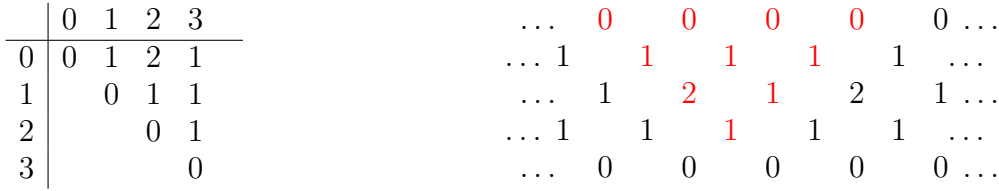
        \begin{align*}
        \begin{array}{c|ccccc}
            &0&1&2&3 \\\hline
            0&0 & 1 & 2 & 1\\
            1& & 0 & 1& 1 \\
            2&&& 0&1\\
            3&&&&0&
        \end{array}
    \end{align*}
    \end{minipage}
    \begin{minipage}{.8\textwidth}
        \begin{align*}
            \begin{matrix}
                \ldots &{\color{red}0}&&{\color{red}0}&&{\color{red}0}&&{\color{red}0}&&0 \; \ldots\\
                \ldots \; 1&&{\color{red}1}&&{\color{red}1}&&{\color{red}1}&&1&\ldots\\
                \ldots &1&&{\color{red}{2}}&&{\color{red}{1}}&&2&&1 \; \ldots\\
                \ldots \; 1&&1&&{\color{red}1}&&1&&1&\ldots\\
                \ldots &0&&0&&0&&0&&0 \; \ldots
            \end{matrix}
        \end{align*}
    \end{minipage}
\end{example}
As for every positive integer $n$, we get a finite frieze from the triangulation of the polygon with vertices $0, \ldots, \yy_n^R$, we have the following lemma by the theory of finite friezes.
\begin{lemma}
    Given a right fountain triangulation $R$ with fountain point 0 and consider the frieze given by the rule \eqref{our-quiddity-rule}. Then, $m_{a,b} = 1$ if and only if $(a,b)$ is an arc in $R$.
\end{lemma}
And an immediate corollary is the following Proposition.
\begin{proposition}
    Given a right fountain triangulation $R$ with fountain point 0 and consider the frieze given by the rule \eqref{our-quiddity-rule}. Put
    \begin{align*}
        x_n^R = \begin{cases}
            1 &\quad \text{if } m_{0,n+1} = 1 \\
            0 & \quad \text{else}.
        \end{cases}
    \end{align*}
    Then, we have $x^R = \xx^R$. If $R$ is a special right fountain triangulation, then the first row of the half-frieze written in the form \eqref{Eq:array-frieze} gives us the sequence $\xx^R$ and hence determines a Penrose tiling.
\end{proposition}
It is clear that in general we cannot reconstruct the first row of the frieze from the sequence $\xx^R$. This aligns with our discussion in Section \ref{Sub:Penrose-triang}. However, in Section \ref{Sub:Fountain-frieze} there is an example of a  frieze that can be uniquely reconstructed from $\xx^R$.

\section{From cluster characters to friezes} \label{Sec:CC->Frieze}

Next we give a representation theoretic explanation of half-friezes by considering a cluster character on the Frobenius category $\mc{C}_2=\CM_{\ZZ}(R)$ introduced in Section \ref{Sub:MCM-Frobenius-cat}. Since most arguments work similarly for the Paquette--Y{\i}ld{\i}r{\i}m category $\mc{C}$ (recalled below), we simultaneously treat the two cases: the main difference is that in the triangulated category $\mc{C}$ there are no non-zero projective injective objects  (corresponding boundary edges in the $\infty$-gon). \\
We first recall how to get from $\CCC_2$ (resp.~$\CCC$) to friezes via the cluster character. Therefore we will describe the Auslander--Reiten quiver for these categories (see Prop.~\ref{Prop:AR-quiver-Frobenius}). 

First recall the Paquette--Y{\i}ld{\i}r{\i}m category $\mc{C}$ \cite{Paquette-Yildirim}, the completed version of the Holm--J{\o}rgensen category $\mc{D}$ from \cite{Holm-Jorgensen}, which is equivalent to the stable version of our Frobenius category $\mc{C}_2$, for details see \cite{ACFGS}:
$$\CCC=\overline{\mc{D}} \simeq \underline{\mc{C}}_{2}=\underline{\CM}_\ZZ(\CC[x,y]/(x^2)) \ . $$ 
We will make use of the combinatorial model for both of these categories: the arcs $\gamma$ in the completed $\infty$-gon. Note that the difference in these two combinatorial models is that for $\mc{C}_2$ we also consider the boundary of the $\infty$-gon. In order to compute the cluster character for an object $M_\gamma \in \mc{C}$ or $\mc{C}_2$ for a cluster tilting subcategory $\mc{T}$ of $\mc{C}$ or $\S$ of $\mc{C}_2$, we first have get from the arc $\gamma$ to the object $M_\gamma$ in $\mc{C}$ or $\mc{C}_2$. This goes as follows: the triangulation $\mc{T}$ consists of some arcs $\{\alpha_i\}_{i \in I}$ in the $\infty$-gon. By a direct computation in the Frobenius case $\mc{C}_2$ \cite[Appendix A.2]{ACFGS} and by \cite[Theorem 4.4 with one accumulation point]{Paquette-Yildirim} in the triangulated case, $\mc{T}$ (resp. $\S$) corresponds to a fountain triangulation of the (completed) $\infty$-gon (see Fig. \ref{Fig:fountain}). \\
 We explain the correspondence in the triangulated case: to $\mc{T}$ one associates a quiver $Q_{\mc{T}}$ in the same way as in \cite{SchifflerBook} (Note: we label vertices on the $\infty$-gon counterclockwise, so we have to be careful how to orient the quiver - we choose to orient  $Q_\mc{T}$ in the same way as in  \cite{SchifflerBook}). Note that $Q_{\mc{T}}$ will have infinitely many vertices, which we may label by positive integers, i.e., we assume that $I=\N_{>0}$. Then the projective $P(i)$ at $i$ is given as  $\tau^{-1}(M_{\alpha_i})$ (and hence the shifts $P(i)[1]$ of the projectives correspond to the objects $M_{\alpha_i}$. So $M_\gamma$ is a module over $kQ_{\mc{T}}$ if and only if $\gamma \not \in \{\alpha_i \}$). Note that in terms of arcs, $\tau^{-1}$ moves the arc one space counterclockwise: if $\gamma=(a,b)$ with $a < b$, then $\tau^{-1}(\gamma)=(a+1,b+1)$. 
  If there is no danger of confusion we will identify the arc $(a,b)$ with the object $M_{(a,b)} \in \mc{C}$. \\

Further, for any arc $\gamma$, the corresponding module $M_\gamma$ is in terms of the quiver $Q_{\mc{T}}$: first compute all the $\alpha_{i}$ that the arc $\gamma$ crosses. Get a representation of $Q_{\mc{T}}$ by putting a vector space $k$ at vertex $i$ (corresponding to $\alpha_i$) and $0$ if $\gamma$ does not cross $\alpha_i$. The maps in the representation are just identities or $0$. If $\gamma \not \in \{\alpha_i\}$, then this yields a non-zero representation of $Q_{\mc{T}}$, and hence a non-zero module over $kQ_{\mc{T}}$. 

\subsection{Extensions and the Auslander--Reiten quiver of $\mc{C}_2$ and $\mc{C}$}

The AR-sequences in $\mc{C}_2$ (resp.~AR-triangles in $\mc{C}$) can then either be computed using knitting or using arcs in the $\infty$-gon. We refer to \cite{ASS, Auslander, Yoshino} for generalities on AR-sequences. 
First recall the extensions with indecomposable end terms in $\mc{C}_2$. Here the dimension of $\Ext(-,-)$ is at most one-dimensional over $\CC$. The only nonsplit short exact sequences (up to scalars) with indecomposable end terms are listed below, cf. \cite[Lemmas 4.5, 4.6, 4.7]{ACFGS}.
\begin{enumerate}
    \item If $a < c < b < d$, then we have non-split short exact sequences
    \begin{align*}
        0 \to (a,b) \to (a,d) \oplus (c,b) \to (c,d) \to 0 \\
        0 \to (c,d) \to (a,c) \oplus (b,d) \to (a,b) \to 0 \ .
    \end{align*}
    \item If $a < b < c$, then we have non-split short exact sequences
    \begin{align*}
        0 \to (b, \infty) \to (a,b) \oplus (c, \infty) \to (a,c) \to 0\\
        0 \to (a,c) \to (a, \infty) \oplus (b,c) \to (b, \infty) \to 0 \ .
    \end{align*}
    \item If $a < b$, then we have a non-split short exact sequence
    \begin{align*}
        0 \to (b, \infty) \to (a,b) \to (a, \infty) \to 0 \ .
    \end{align*}
\end{enumerate}

\begin{lemma} \label{Lem:finiteARseq}
For any finite arc $(a,b) \in \mc{C}_2$ which is not a boundary arc (\textit{i.e.} $b-a \geq 2$) the short exact sequence
    \begin{align} \label{AR-sequence-Frobenius}
    0 \to (a-1,b-1) \to (a-1,b) \oplus (a,b-1) \to (a,b) \to 0  
    \end{align}
is an almost split sequence (i.e., the AR-sequence) ending at $(a,b)$. Note that the AR-translate of $(a,b)$ is $\tau((a,b))=(a-1,b-1)$. \\
The corresponding exchange sequence is
\begin{equation}  \label{eq:exchangeseq} 0 \to (a,b) \to (a-1,a) \oplus (b-1,b) \to (a-1,b-1) \to 0 
\end{equation}
Similarly, the short exact sequence
    \begin{align} 
    0 \to (a,b) \to (a+1,b) \oplus (a,b+1) \to (a+1,b+1) \to 0  
    \end{align}
is an almost split sequence (i.e., the AR-sequence) starting at $(a,b)$.
 \\
For an infinite arc $(a,\infty)$ there does not exist an almost split sequence starting or ending with it.
\end{lemma}

\begin{proof}
The claims for finite arcs can be readily verified by looking at the description of extensions above and computing the minimal ones. The statement for infinite arcs can be also be verified by direct computation by using morphisms from \cite{ACFGS} and extensions listed above.
\end{proof}

\begin{lemma}  \label{Lem:infARseq}
Let $(a, \infty)$, $(b, \infty)$ be infinite arcs in $\CCC_2$. A homomorphism $(a,\infty) \xrightarrow{} (b, \infty)$ factors if and only if $a \leq b$ and $ a \neq b-1$. In other words: the only irreducible maps between infinite arcs are $(a-1, \infty) \xrightarrow{} (a, \infty)$. \\
Moreover, any homomorphism from a finite arc $(c,d)$ to $(a, \infty)$ factors, and any homomorphism from $(a,\infty)$ to a finite arc $(c,d)$ factors. In other words: there are no irreducible maps between infinite arcs and finite arcs.
\end{lemma}

\begin{proof}
This is again verified by direct computation using the description of the morphisms from \cite[Lemma A.3, A.4]{ACFGS}: in particular, the irreducible maps $(a-1, \infty) \xrightarrow{} (a,\infty)$ are given by $ 1 \mapsto \lambda y$ for some $\lambda \in \CC$, and any non-zero map $(a, \infty) \xrightarrow{} (a, \infty)$ is an isomorphism.
\end{proof}

\begin{proposition} \label{Prop:AR-quiver-Frobenius}
The Auslander--Reiten quiver of $\CCC_2$ looks like

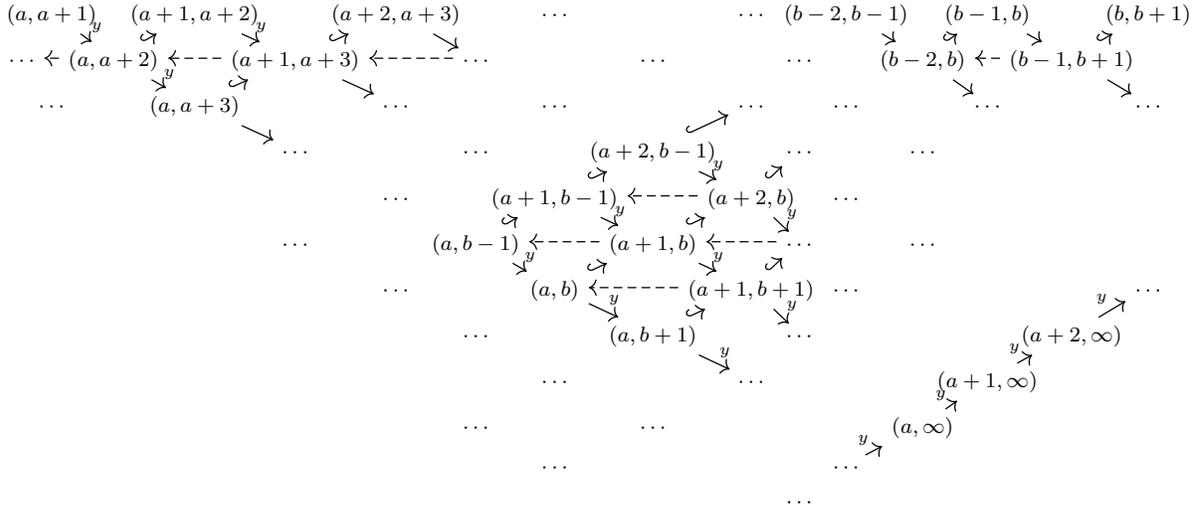
\[
{\tiny 
\begin{tikzcd}[row sep=0.4em, column sep=-2em]
 &   && (a,a+1) \ar[rd,"y"]  && (a+1,a+2)  \ar[rd,"y"] && (a+2,a+3)  \ar[rd]&& \cdots && \cdots && (b-2,b-1)  \ar[rd] && (b-1,b) \ar[rd] && (b,b+1) \\ 
&& \cdots &&(a,a+2) \ar[rd, "y"] \ar[ll, dashed] \arrow[ru,hook]&& (a+1,a+3) \arrow[ru, hook] \ar[rd] \ar[ll,dashed] && \ar[ll, dashed] \cdots && \cdots && \cdots && (b-2,b) \arrow[ru,hook] \ar[rd]&& \ar[ll, dashed] (b-1,b+1) \ar[ur,hook] \ar[rd] &&    \\  
 & && \cdots && (a,a+3)\arrow[ru, hook] \ar[rd]&& \cdots && \cdots  && \cdots &&  \cdots && \cdots &&  \cdots  \\ 
&& && && \cdots && \cdots && (a+2,b-1) \ar[rd, "y"]  \arrow[ru,hook] && \cdots && \cdots && &&   \\ 
&  &&  &&  && \cdots &&  (a+1,b-1) \ar[rd, "y"]  \arrow[ru, hook] && (a+2,b) \ar[rd, "y"] \ar[ll,dashed] \arrow[ru,hook]&& \cdots &&  &&     \\ 
&&  &&  && \cdots && (a,b-1) \arrow[ru, hook] \ar[rd, "y"] && (a+1,b) \arrow[ru, hook]  \ar[rd, "y"] \ar[ll,dashed] && \ar[ll,dashed] \cdots &&\cdots && &&      \\ 
& &&  &&  && \cdots && (a,b) \arrow[ru, hook] \ar[rd, "y"] && (a+1,b+1) \arrow[ru, hook] \ar[ll,dashed] \ar[rd, "y"]&& \cdots  && &&  \cdots &&    \\ 
&&  && &&  && \cdots && (a,b+1) \arrow[ru, hook] \ar[rd,"y"] && \cdots  && &&  (a+2, \infty) \arrow[ru,"y"] &&  &&    \\ 
&  &&  &&  &&  &&  \cdots && \cdots && && (a+1,\infty) \arrow[ru,"y"]  && &&   \\ 
&&  &&  &&  && \cdots && \cdots && && (a,\infty) \arrow[ru, "y"] && &&  &&  \\ 
&  &&  &&  &&  &&  \cdots && && \cdots \arrow[ru, "y"] &&  && &&   \\ 
&& \ && && &&  &&  && \cdots && &&  &&  &&   \\
\end{tikzcd}
}
\]

Thus, we see that the AR-quiver has two infinite components. In the component consisting of the infinite arcs none of them has an AR-translation.
\end{proposition}

\begin{proof}
The shape of the component consisting of all the finite arcs follows from Lemma \ref{Lem:finiteARseq}. From Lemma \ref{Lem:infARseq} follow the shape of the component involving the infinite arcs, that there are no irreducible maps between the two components, and that infinite arcs do not have an AR-translation.
\end{proof}

\begin{remark}
The AR-quiver for the triangulated subcategory of $\mc{C}_2$ consisting of the non-projective finite arcs, i.e., the Holm--J{\o}rgensen category $\mc{D}$, has already been computed in \cite[Remark 1.4]{Holm-Jorgensen} as $\ZZ A_\infty$. From another point of view, Igusa and Todorov also computed the AR-quiver of their category in \cite[Lemma 2.4.11]{IgusaTodorov-cyclic}. Note that if we take the generically free modules in our category $\mc{C}_2$ and stabilise, we recover their results. \\
Further note that the AR-quiver of the ungraded category $\CM(\CC\llbracket x,y \rrbracket / (x^2))$ has been computed in the context of singularity theory in \cite{Schreyer} (also cf. \cite{Leuschke-Wiegand, Yoshino}).
\end{remark}

\begin{remark} 
For completeness, we also indicate the AR-triangles in the category $\mc{C}$: For a finite arc $\gamma=(a,b)$ with $a <b-1$ we have  
the AR triangle, where again we write shorthand $(a,b)$ for the object $M_{(a,b)}$:
\begin{equation}(a-1,b-1) \xrightarrow{} (a-1,b)\oplus (a,b-1) \xrightarrow{} (a,b) \xrightarrow{} (a-1,b-1)[1]. \label{ARtriangle-stable}\end{equation}

Further, in $\mc{C}$ the triangle \eqref{ARtriangle-stable} for $b=a+2$ has a different form (since the only  projectives are $0$):
\begin{equation} (a-1,a+1) \xrightarrow{} (a-1,a+2) \xrightarrow{} (a,a+2) \xrightarrow{} (a-1,a+1)[1]. \label{ARtriangle-stable-rand}\end{equation}

Moreover, in $\mc{C}$ we do not have the projectives, so the exchange sequence for any finite arc $(a,b)$ is  
\begin{equation} \label{eq:exchangetriangle}(a,b) \xrightarrow{} 0  \xrightarrow{} (a-1,b-1) \xrightarrow{} (a,b)[1] \ ,
\end{equation}
and in terms of objects in $\mc{C}$ we have $M_{(a-1,b-1)} \cong M_{(a,b)}[1]$, see \cite[Lemma 1]{CalderoKellerI}.
\end{remark}

\subsection{Paquette--Y{\i}ld{\i}r{\i}m cluster character for $\mc{C}$} \label{PY-cluster-char}

Paquette--Y{\i}ld{\i}r{\i}m define a cluster character on $\mc{C}$ with respect to a cluster tilting subcategory $\mc{T}$ for any $M \in \mc{C}$ via the formula
\begin{equation} \label{eq:PY-clusterchar} 
\chi_\mc{T}(M)=x^{-\coind(\varphi(M))}\sum_{d}\chi(\mathrm{Gr}_d(\varphi(M)))x^{\coind(d)-\ind(d)} \ .
\end{equation}
Here $\varphi: \mc{C} \xrightarrow{} \textrm{Mod} \mc{T}$ is the covariant functor that sends $M$ to $\Hom_\mc{C}(-,M)$ and $d$ runs over dimension vectors of finitely presented submodules of $\varphi(M)$, see \cite[Section 5]{Paquette-Yildirim}. Index and Co-index are defined in \cite[Section 6.2]{Paquette-Yildirim}. \\
Note that the variables $x$ are in the homogeneous coordinate ring of the infinite Grassmannian $Gr(2,\infty)$ with the frozen  variables $x_{i,i+1}$ set to $1$, namely $\CC[x_{ij}: i,j \in \ZZ]/($Pl\"ucker relations$)$, where $\deg(x_{ij})=1$ and $i<j-1$. See \cite[Appendix]{Grabowski-Gratz-Groechenig} for more details. Choosing the cluster tilting subcategory $\mc{T}$ means to choose a fountain triangulation in the $\infty$-gon. Thus, we have variables $x_{ab}$, where $(a,b)$ corresponds to an arc in $\mc{T}$ (we may assume here that $a<b$). For  any $T \in \mc{T}$ we denote by $[T]$ the corresponding class in $K_0(\mc{C})$, and if $[T]=\sum_{i \in I}m_i[T_i]$, then $T_i$ is an indecomposable corresponding to an arc $(a_i,b_i)$ in $\mc{T}$. Thus here we mean with $x^{[T]}=\prod_{i \in I}x_{a_ib_i}^{m_i}$. Note that $\chi_{\mc{T}}(M)$ is an element of $\CC\llbracket x_{a,b}^\pm: (a,b) \in \mc{T} \rrbracket$.

In \cite[Prop.~6.9]{Paquette-Yildirim} it is shown that $\chi_{\mc{T}}(M)$ satisfies the product formula, that is, 
$$\chi_{\mc{T}}(M_1 \oplus M_2)=\chi_{\mc{T}}(M_1) \cdot \chi_{\mc{T}}(M_2) $$
for any $M_1, M_2 \in \mc{C}$. However, since $\mc{C}$ is not $2$-CY, the exchange formula for the cluster character only holds on the subcategory of $\mc{C}$ corresponding to the finite arcs in the completed $\infty$-gon, see \cite[Theorem 6.13]{Paquette-Yildirim}.

\subsection{Cluster character for $\mc{C}_2$}

Now we switch to the Frobenius category $\mc{C}_2$ and show that there is a cluster character too, see Def. \ref{Def:ourCC}. In Section \ref{Sub:relations} we will then show that our cluster character satisfies the generalised frieze relations, matching up with (infinite) frieze patterns with coefficients in the sense of Cuntz, Holm, and J{\o}rgensen \cite{CHJ-coefficients}, see Theorem \ref{Thm:half-friezes-coeffs}.

\begin{definition} \label{Def:ourCC}
Let $\S$ be a cluster tilting subcategory of $\mc{C}_2$ and let $M \in \mc{C}_2$ an object. Then the cluster character is 
\begin{equation} \label{Eq:CC-Frobenius}
X^\S(M) := x^{\ind_\S(M)} \sum_{e} \chi\left(\mathrm{Gr}_{e}(\Ext^1_{\mc{C}_2}(-,M))\right) x^{\coind_\S(e)-\ind_\S(e)} \ ,
\end{equation}
where the sum runs again over all  finitely presented dimension vectors $e$ in $\Ext^1(-,M)$. Note that $X^\S(M) \in \CC\llbracket x_{a,b}^\pm: (a,b) \in \S \rrbracket=:R_\S$. Furthermore, since all the boundary arcs $(n,n+1)$ have to be contained in $\S$, we have that $x_{n,n+1}$ are variables in $R_\S$ and $X^\S((n,n+1))=x_{n,n+1}$ (cf.~Lemma \ref{Lem:CC-frieze-coeff}) are precisely the frozen variables in the homogeneous coordinate ring of $Gr(2, \infty)$.
\end{definition}

Our cluster character follows the general principle of defining a cluster character for a 2-Calabi-Yau extriangulated category equipped with a cluster tilting object as outlined in \cite[Section 4.3]{wang-wei-zhang}. The only and the main obstacle is that our cluster tilting subcategories are not additively generated by a single object, they have infinitely many indecomposable objects up to isomorphisms. This situation, in the triangulated case, is treated carefully by Paquette and Y{\i}ld{\i}r{\i}m. We note the isomorphisms 
\begin{align*}
    \Ext_{\mc{C}_2}^1(X,M) \cong \Hom_{\underline{\mc{C}_2}}(\Omega X, M) \cong \Hom_{\underline{\mc{C}_2}}(X, \Omega^{-1}M) \cong \Hom_{\mc{C}}(\pi X, \pi M[1])
\end{align*}
where $\pi$ denotes the composition 
\begin{align*}
\mc{C}_2 \to \underline{\mc{C}_2} \xrightarrow{\sim} \mc{C}.
\end{align*}
Hence, the Grassmannian in our cluster character is meaningful: it agrees with the Grassmannian defined by Paquette and Y{\i}ld{\i}r{\i}m when the object $M$ is not projective. When $M$ is projective, its image in the stable category is isomorphic to the zero object and so is its shift. So, the Euler characteristic returns zero. As every projective object belongs to the cluster tilting subcategory $\S$, this aligns with Lemma \ref{Lem:CC-frieze-coeff}(a) below.

Secondly, we note that there is an equality $\ind_S(M) = -\coind_S(\Omega^{-1}M)$ which also proves that for a non-projective object $M$ we have an equality (cf. \cite[Theorem 3.3(b)]{Keller-Fu} and \cite[Proposition 4.5]{wang-wei-zhang}) 
\begin{align*}
X^\S(M) = \chi_{\pi S}(\pi M [1]) \  ,
\end{align*}
where $\pi: \mc{C}_2 \rightarrow \mc{C}$ is the projection on the stable category.
From these two remarks, it follows that our formula does indeed give a cluster character, when restricted to generically free maximal Cohen--Macaulay modules. Here we note that the subcategory of generically free maximal Cohen--Macaulay modules in $\mc{C}_2$ is indeed stably $2$-CY, see \cite[Prop.~3.12]{August-Cheung-Faber-Gratz-Schroll}.

\subsection{The frieze relations} \label{Sub:relations}
Here we prove that the cluster character $X^{\S}$ for $\mc{C}_2$ and $\chi_{\mc{T}}$ for $\mc{C}$ for $\S \subseteq \mc{C}_2$, $\mc{T} \subseteq \mc{C}$ cluster tilting subcategories satisfies the frieze relations on the subcategory corresponding to finite arcs in $\mc{C}_2$, resp. $\mc{C}$. Then we define infinite (half-)friezes (with coefficients). Finally, we show that indeed the cluster character on $\mc{C}_2$ and $\mc{C}$ yields infinite (half-)friezes in Theorem \ref{Thm:half-friezes-coeffs}.

\begin{remark} \label{Sub:triangulations}
Here we quickly recall the relation between cluster tilting subcategories of our categories $\CCC$ and $\CCC_2$ and triangulations of the completed $\infty$-gon: any cluster tilting subcategory of $\CCC$ or $\CCC_2$ corresponds to a fountain triangulation of the $\infty$-gon, see \cite[Theorem 4.4]{Paquette-Yildirim} for the triangulated case and \cite[Theorem 4.11]{ACFGS} for the Frobenius case, also cf.~Section \ref{Sub:comb-model}. \\
If one restricts to the subcategory of generically free modules $\CCC_2^f$ of $\CCC_2$ or the Holm--J{\o}rgensen category in $\CCC$ (that is, finite arcs in the $\infty$-gon), then the cluster tilting subcategories correspond to either fountain triangulations or locally finite triangulations (so-called leapfrog triangulations), see \cite[Theorem 4.10]{ACFGS}. Also see Fig.~\ref{Fig:fountain} for a schematic picture of a fountain triangulation and the upper left triangulation in Fig.~\ref{Fig:Triangulations} for a schematic picture of a leapfrog triangulation. \\
Furthermore, in Section \ref{Sub:comb-model} schematic pictures of all triangulations of the completed $\infty$-gon are shown.
\end{remark}

\begin{lemma} \label{Lem:CC-frieze}
Consider the stable category $\mc{C}$ together with a cluster tilting subcategory $\mc{T}$ and the cluster character $\chi_\mc{T}$ from Section \ref{PY-cluster-char}. Let $\{ \alpha_i \}_{i \in I}$ be the triangulation of the $\infty$-gon corresponding to $\mc{T}$.
\begin{enumerate}[(a)]
\item If $\alpha_i$ is in the triangulation, that is, the indecomposable $M_{\alpha_i}$ is in $\mc{T}$, then $\chi_{\mc{T}}(M_{\alpha_i}[1])=1$.
\item The cluster character satisfies the frieze relation
\begin{equation} \label{eq:frieze-middle}
\chi_{\mc{T}}(M_{(a-1,b-1)}) \cdot \chi_{\mc{T}}(M_{(a,b)}) = \chi_{\mc{T}}(M_{(a-1,b)}) \cdot \chi_{\mc{T}}(M_{(a,b-1)}) +1 
\end{equation}
for any integers $a +2< b$, and  
\begin{equation} \label{eq:frieze-rand}
 \chi_{\mc{T}}(M_{(a-1,a+1)}) \cdot \chi_{\mc{T}}(M_{(a,a+2)}) = \chi_{\mc{T}}(M_{(a-1,a+2)}) +1 
 \end{equation}
for $b=a+2$.
\end{enumerate}
\end{lemma}

\begin{proof} First note that the cluster character of $\mc{C}$ with respect to $\mc{T}$ for a finitely presented module $M$ is given by the formula \eqref{eq:PY-clusterchar}. Property (a) follows from substituting $M_{\alpha_i}[1]$ into this formula: since $M_{\alpha_i} \in \mc{T}$, we have $\varphi(M_{\alpha_i}[1])=0$ and thus the product gives $1$ and since $\coind(0)=0$, \eqref{eq:PY-clusterchar} yields $\chi_\mc{T}(M_{\alpha_i}[1])=1$.

(b) follows from \eqref{ARtriangle-stable} resp. \eqref{ARtriangle-stable-rand} and \eqref{eq:exchangetriangle} and using the exchange and the multiplication property of the cluster character (see \cite[Section 6.2]{Paquette-Yildirim}): consider $M_{(a,b)}$ with $a+2<b$, then the AR-triangle of this module and its exchange sequence give
$$ \chi_\mc{T}(M_{(a,b)}) \cdot \chi_{\mc{T}}(M_{(a-1,b-1)})=\chi_{\mc{T}}(M_{(a-1,b)}) \cdot \chi_{\mc{T}}(M_{(a,b-1)})+1 \ ,$$
since $\chi_{\mc{T}}(0)=1$. Similarly, for $a+2=b$, using \eqref{ARtriangle-stable-rand},  relation \eqref{eq:frieze-rand} follows. 
\end{proof}

\begin{lemma} \label{Lem:CC-frieze-coeff} 
Consider the Frobenius category $\mc{C}_2$ together with a cluster tilting subcategory $\S$ and our cluster character $X^\S$ from Def.~\ref{Def:ourCC}.
\begin{enumerate}[(a)]
\item If $\alpha=(a,b)$ is in the triangulation corresponding to $\S$, that is, $M_{\alpha}$ is in $\S$, then $X^\S(M_{\alpha})=x_{a,b}$.
\item The cluster character satisfies the generalised frieze relation
\begin{equation} \label{eq:frieze-middle-coeff}
X^\S(M_{(a-1,b-1)}) \cdot X^\S(M_{(a,b)}) = X^\S(M_{(a-1,b)}) \cdot X^\S(M_{(a,b-1)}) + x_{a-1,a}\cdot x_{b-1,b} 
\end{equation}
for any integers $a +1< b$. 
\end{enumerate}
\end{lemma}

\begin{proof}
The first part follows by plugging $M_{(a,b)}$ for an indecomposable $M_{(a,b)} \in \S$ into expression \eqref{Eq:CC-Frobenius}: $M_{(a,b)}$ is rigid and  moreover, we have $x^{\ind_\S(M_{(a,b)})}=x_{a,b}$, since $(a,b)$ is in the triangulation of the $\infty$-gon corresponding to $\S$. Thus $X^\S(M_{(a,b)})=x_{a,b}$.

Part (b) works analogously to the proof of (b) of Lemma \ref{Lem:CC-frieze} using the exact sequences for the Frobenius category \eqref{AR-sequence-Frobenius} and \eqref{eq:exchangeseq}, as well as the multiplication and exchange rules for the cluster character (see \cite[Theorem 4.4]{wang-wei-zhang}) and (a) for $X^\S(M_{a,a+1})=x_{a,a+1}$ for any boundary arc $(a,a+1)$.
\end{proof}

The previous Lemma motivates the definition of infinite friezes with coefficients, adapting the finite definition of \cite[Def.~2.1.]{CHJ-coefficients}.
\begin{definition} \label{Def:inf-frieze-coeff}
An \emph{infinite frieze with coefficients} is an  infinite array of elements $m_{a,b}$ in a commutative ring $R$ with $a,b  \in \ZZ$ and $a \leq b$ of the form \eqref{Fig:inf-frieze-coeff} satisfying
\begin{enumerate}
\item $m_{a,a}=0$ for all $a \in \ZZ$.
\item $m_{a,a+1}=c_{a,a+1} \in R\backslash \{0
\}$ for all $a \in \ZZ$.
\item $m_{a,b} m_{a+1,b+1} - m_{a,b+1}m_{a+1,b}=c_{a,a+1} c_{b,b+1}$ for any $a < b \in \ZZ$
\end{enumerate}
\begin{equation} \label{Fig:inf-frieze-coeff}
{\tiny
\begin{tikzcd}[row sep=0.35em, column sep=-0.95em]
  & 0 &&0 && 0 && 0 && 0&& 0  &&0&&  0&& 0 && 0 && 0 &&  \\ 
&& c_{-5,-4} && c_{-4,-3} && c_{-3,-2} && c_{-2,-1} && c_{-1,0} &&c_{0,1} &&c_{1,2} && c_{2,3} && c_{3,4} && c_{4,5} && c_{5,6}    \\  
  & m_{-6,-4} && m_{-5,-3} && m_{-4,-2} && m_{-3,-1} && m_{-2,0}&& m_{-1,1}  &&m_{02}&&  m_{13}&& m_{24} && m_{35} && m_{46} &&  \\ 
&& m_{-6,-3} && m_{-5,-2} && m_{-4,-1} && m_{-3,0} && m_{-2,1}  && m_{-1,2} && m_{03} && m_{14} && m_{25} && m_{36} && m_{47} &&   \\ 
& m_{-7,-3} && m_{-6,-2} && m_{-5,-1} && m_{-4,0} &&  m_{-3,1} && m_{-2,2} && m_{-1,3} && m_{04} && m_{15} && m_{27} && m_{38} && m_{49} &&   \\ 
&& m_{-7,-2} && m_{-6,-1} && m_{-5,0} && m_{-4,1} && m_{-3,2}  && m_{-2,3}  && m_{-1,4} && m_{05} &&  m_{26} && m_{37} && m_{48} &&    \\ 
& m_{-8,-2} && m_{-7,-1}  && m_{-6,0} && m_{-5,1} && m_{-4,2}  && m_{-3,3}&& m_{-2,4}  && m_{-1,5}  && m_{06} &&  m_{17} && m_{28}  && m_{39} && \\ 
&& m_{-8,-1} &&m_{-7,0}&& m_{-6,1} && m_{-5,2} && m_{-4,3} && m_{-3,4}  && m_{-2,5} && m_{-1,6}  && m_{07} && m_{18} &&m_{29}&&      \\ 
&  m_{-9,-1} && m_{-8,0} && m_{-7,1} && m_{-6,2} && m_{-5,3}&& m_{-4,4}&& m_{-3,5} && m_{-2,6} && m_{-1,7} &&m_{08} && m_{19} && m_{2,10} &&  \\ 
&& \ddots && \ddots && \ddots && \ddots && \ddots && \ddots && \ddots && \ddots && \ddots && \ddots && \ddots \\
\end{tikzcd}
}
\end{equation}
Similarly as in Definition \ref{Def:infinite-half-frieze} we may define a   \textit{(right) half-frieze with coefficients} (centered at $0$)  as an array of elements $m_{a,b} \in R$, $a \leq b$ such that
    \begin{enumerate}
        \item $m_{a,a} = 0$ for any $a \geq 0$.
        \item $m_{a,a+1} = c_{a,a+1} \in R\backslash\{0\}$ for any $a \geq 0$.
        \item $m_{a,b} \; m_{a+1, b+1} - m_{a+1, b} \;m_{a, b+1} = c_{a,a+1} \: c_{b,b+1}$ for any pair $(a,b)$ with $0 \leq a \leq b$.
    \end{enumerate}
Analogously one defines left half-friezes with coefficients. Putting together a left half-frieze and a right half-frieze with coefficients centered at $0$, we obtain a \emph{fountain frieze with coefficients}: \\

\begin{equation} \label{Fig:half-frieze-coeff}
{\tiny
\begin{tikzcd}[row sep=0.35em, column sep=-0.75em]
  & 0 &&0 && 0 && 0 && 0&& 0  &&0&&  0&& 0 && 0 && 0 &&  \\ 
&& c_{-5,-4} && c_{-4,-3} && c_{-3,-2} && c_{-2,-1} && c_{-1,0} &&c_{0,1} &&c_{1,2} && c_{2,3} && c_{3,4} && c_{4,5} && c_{5,6}    \\  
  & m_{-6,-4} && m_{-5,-3} && m_{-4,-2} && m_{-3,-1} && m_{-2,0}&& \phantom{m}  &&m_{02}&&  m_{13}&& m_{24} && m_{35} && m_{46} &&  \\ 
&& m_{-6,-3} && m_{-5,-2} && m_{-4,-1} && m_{-3,0} &&  &&  && m_{03} && m_{14} && m_{25} && m_{36} && m_{47} &&   \\ 
& m_{-7,-3} && m_{-6,-2} && m_{-5,-1} && m_{-4,0} &&  &&  &&  && m_{04} && m_{15} && m_{27} && m_{38} && m_{49} &&   \\ 
&& m_{-7,-2} && m_{-6,-1} && m_{-5,0} &&  &&  &&  &&  && m_{05} &&  m_{26} && m_{37} && m_{48} &&    \\ 
& m_{-8,-2} && m_{-7,-1}  && m_{-6,0} &&  &&  && &&  &&  && m_{06} &&  m_{17} && m_{28}  && m_{39} &&  \\ 
&& m_{-8,-1} &&m_{-7,0}&&  &&  &&  &&  && &&  && m_{07} && m_{18} &&m_{29}&&     \\ 
&  m_{-9,-1} && m_{-8,0} && && && && &&  &&  &&  &&m_{08} && m_{19} && m_{2,10} &&  \\ 
&& m_{-9,0} && &&  &&  &&  &&  && && &&  && m_{08} && m_{19} && &&      \\ 
\end{tikzcd}
}
\end{equation}
When all the $c_{ab}=1$ for $a <b$, then we are back to the  (half-)friezes resp.~ fountain friezes of Definition \ref{Def:infinite-half-frieze}.
\end{definition}

\begin{remark}
We can also define friezes of different shapes for the other possible triangulations of the $\infty$-gon: \emph{split fountain friezes} consist of a left half frieze centered at some $a$ and a right half frieze centered at some $b >a$ together with a finite CC-frieze in the middle. It would be interesting to study these friezes in future work.
\end{remark}

\begin{definition} \label{Def:shift-frieze}

Let $\mc{F}$ be a infinite (half-)frieze (with coefficients) with entries $m_{a,b}$ as in Definition \ref{Def:inf-frieze-coeff}. Then we denote by $\mc{F}[i]$ the \emph{ $i$-th shift} of the (half-)frieze (with coefficients), which is again a (half-)frieze (with coefficients) with entries $m'_{a,b}$ satisfying 
\[ m'_{a,a+1}=c_{a-i,a-i+1}, \  m'_{a,b}=m_{a-i,b-i} \  \] 
for appropriate $a, b \in \ZZ$.
\end{definition}

\begin{theorem} \label{Thm:half-friezes-coeffs}
\begin{enumerate}
\item
Let $\mc{T}$ be a cluster tilting subcategory of the triangulated category $\mc{C}$. Then the cluster character $\chi_\mc{T}$ \eqref{eq:PY-clusterchar} yields an infinite fountain frieze $\mc{G}$ as in Definition \ref{Def:inf-frieze-coeff}, but without the rows of $0$'s and $c_{ab}$'s on top. In particular, $\chi_\mc{T}(M_{(a,b)}[1])=1$ for any indecomposable $M_{(a,b)} \in \mc{T}$. The remaining entries in this half-frieze can be calculated using Lemma \ref{Lem:CC-frieze}. 
\item
Let $\S$ be a cluster tilting subcategory of the Frobenius category $\mc{C}_2$. Then the cluster character $X^\S$ yields an infinite fountain frieze $\mc{F}$ with coefficients: $\S$ contains all projective-injectives, that is, $M_{(i,i+1)}$ for $i \in \ZZ$, so we have $X^\S(M_{(i,i+1)})=x_{i,i+1}$ and for any arcs $M_{(a,b)} \in \S$ (wlog $a<b$) we also have $X^\S(M_{(a,b)})=x_{a,b}$.  The remaining  entries in this fountain frieze with coefficients can be calculated with Lemma \ref{Lem:CC-frieze-coeff}. \\
Further, taking all $M_{(a,b)} \in \S$ with $a<b-1$ yields a cluster tilting subcategory of $\mc{C}$ with fountain frieze $\mc{G}$ as in (1). Setting $x_{a,b}=1$ for any indecomposable $M_{(a,b)} \in \S$ in the frieze $\mc{F}$, we obtain the fountain frieze $\mc{G}[1]$ of (1) (that is, the entries of the half frieze of (1) shifted one place to the left).
\end{enumerate}
\end{theorem}

\begin{proof}

A cluster tilting subcategory $\T$ in $\mc{C}$ corresponds to a fountain triangulation of the completed $\infty$-gon, see Remark \ref{Sub:triangulations}, w.l.o.g. at $0$.
First let us note that although an infinite arc is contained in the triangulation corresponding to $\T$, this arc will not be part of the frieze: from the description of the AR-quiver of $\CCC_2$ in Proposition \ref{Prop:AR-quiver-Frobenius} we see that the exchange formula only makes sense for the modules corresponding to finite arcs. Thus we just restrict to this part of the AR-quiver (that is, the part of type $\ZZ  A_{\infty}$). 
From Lemma \ref{Lem:CC-frieze} (1) we know that $\chi_{\mc{T}}(T[1])=1$ for all $T \in \mc{T}$.  For any fountain arc  $(0,n)$ with $n \geq 2$  one can use Lemma \ref{Lem:CC-frieze} to compute all entries $m_{a-1,b-1}$ of $\mc{G}$ with $0 \leq a < b \leq n$ of $\mc{G}$.

Since there are infinitely many arcs $(0,n) \in \mc{T}$, we can compute all entries of the positive part of the half frieze. The same procedure for $(-n,0)$, $n \geq$ yields the remaining negative part of the half frieze.

Assertion (2) follows analogously from Lemma \ref{Lem:CC-frieze-coeff}, using the cluster character $X^\S$.
For the last assertion note that for a cluster tilting subcategory $\mc{T} \in \mc{C}$, the corresponding cluster tilting subcategory $\S$ in $\mc{C}_2$ is $\S=\mc{T} \cup \{(n,n+1): n \in \ZZ\}$. Further, we have $\chi_{\mc{T}}(M_{(a,b)}[1])=X^\S(M_{(a,b)})$. Thus the frieze from the Frobenius category is $\mc{F}=\mc{G}[1]$.
\end{proof}

\section{Examples} \label{Sec:Examples}

\subsection{Fountain triangulation} \label{Sub:Fountain-frieze}For ease of comparison with the computations in \cite[Example 6.14]{Paquette-Yildirim}, we  use the same notation. First we will compute the fountain frieze consisting of a left and a right half frieze for the fountain triangulation 
$$\S=\{ (0,n): |n| \geq 2 \} \cup \{ (n,n+1): n \in \ZZ \} \cup \{ (0, \infty)\}$$ 
in $\mc{C}_2$. In the notation of Paquette--Y{\i}ld{\i}r{\i}m, this means that the triangulation consists of arcs $\alpha_i=(0,n)$ and $\beta_i=(-n,0)$ for $n \geq 2$. We will use our cluster character $X^\S$ and restrict the variables $x_{a,b}$ for $(a,b) \in \S$ to $1$. Then we get that $X^\S(\alpha_i)=X^\S(\beta_i)=1$, as well as the cluster characters of the boundary arcs are all $1$. The calculations are the same as for the Paquette--Y{\i}ld{\i}r{\i}m-cluster character for the cluster tilting subcategory $\mc{T}=\S \setminus \{ (n,n+1): n \in \ZZ \}$ of $\mc{C}$, see \cite[Examples 6.14 and 6.15]{Paquette-Yildirim}, except that $X^\S(\alpha)=\chi_\mc{T}(\alpha[1])$ for any $\alpha$ in $\mc{C}$. The beginning of the frieze looks as follows: 

   \begin{align*}
    \begin{array}{c|cccccccccc}
        &0&1&2&3&4&5&6&7&8&\ldots \\\hline
        0&0 & 1 & 1 & 1 & 1 & 1&1 & 1 & 1 & \ldots \\
        1& & 0 & 1& * & * & * & *& *& *& \ldots \\
        2&&& 0&1& * &* &*&*&*& \ldots \\
        3&&&&0&1& *&*&*&*&\ldots \\
        4&&&&&0& 1&*&*&*& \ldots \\
        5&&&&&&0&1&*&*&\ldots\\
        6&&&&&&&0&1&*&\ldots \\
        7&&&&&&&&0 & 1&\ldots
    \end{array}
\end{align*}

Now the entries $(i,j)$ with either $0 \leq i < j$ or $i < j \leq 0$ can be computed with Lemma \ref{Lem:CC-frieze-coeff}, that is, the half frieze coming from the fountain triangulation. Here is a picture of the first part of it

   \begin{align*}
    \begin{array}{c|cccccccccc}
        &0&1&2&3&4&5&6&7&8&\ldots \\\hline
        0&0 & 1 & 1 & 1 & 1 & 1&1 & 1 & 1 & \ldots \\
        1& & 0 & 1& 2& 3& 4 & 5 & 6 & 7 & \ldots \\
        2&&& 0&1& 2 &3 &4&5&6& \ldots \\
        3&&&&0&1& 2&3 &4 &5 &\ldots \\
        4&&&&&0& 1&2 &3&4& \ldots \\
        5&&&&&&0&1&2&3&\ldots\\
        6&&&&&&&0&1&2&\ldots \\
        7&&&&&&&&0 & 1&\ldots
    \end{array}
\end{align*}

One can also compute the cluster character $X^S(\eta) \in  \CC\llbracket x_{a,b}^\pm: (a,b) \in \S \rrbracket$ for any arc $\eta$ that \emph{crosses} the fountain, that is, $\eta \cap \alpha_i \neq \varnothing$ for infinitely many $i$.  But the result will be a Laurent series with positive coefficients in the $x_{a,b}$, see for example the calculation in \cite[Examples 6.14 and 6.15]{Paquette-Yildirim}. Thus specialising the variables to $1$ will yield a non-integer expression, i.e., $\infty$. Note that this fountain frieze does not correspond to a Penrose tiling, since the sequence $\xx^R$ is the constant sequence $1,1,1,1\ldots$. One would need infinitely many mutations to obtain a valid index sequence for a Penrose tiling.

For an example of a frieze that is coming from the periodic index sequence $\xx^R=1,0,1,0, \ldots$, see \cite[Example 4.20]{FaberICRA}. In this example, the special fountain triangulation $\S$ of the completed $\infty$-gon is unique. It would be interesting to further study connections between (periodic) index sequences and their corresponding friezes and triangulations.

\subsection{Leapfrog triangulation}
Here we work in $\mc{C}_2$ with 
$$\S=\{ (-n,n-1): n \in \N_{\geq 2}\} \cup \{ (-n,,n): n \in \N_{\geq 1} \} \cup \{ (n,n+1): n \in \ZZ\} \ .$$
Although $\S$ is not cluster tilting in $\mc{C}_{2}$, it is still maximally rigid, see \cite[Appendix A.2]{ACFGS}. Note that if we consider the projection of $\S$ to the Holm--J{\o}rgensen category $\mc{D}$, then this is indeed cluster tilting. Hence we can compute the cluster character $X^\S$ for all generically free indecomposable modules in $\mc{C}_2$, that is, all finite arcs. Setting $x_{a,b}=1$ for any $(a,b) \in \S$, we thus obtain a non-periodic infinite frieze, since we can compute all entries in the frieze \eqref{Fig:leapfrogfrieze} with Lemma \ref{Lem:CC-frieze-coeff}. Note that the quiddity sequence is 
$$\{ a_i\}_{i=-\infty}^\infty = \begin{cases} 1 \text{ if } i=0 \\
  								2 \text{ if } i=1 \\
								3 \text{ else}
								\end{cases}
								$$

  \begin{align}
    \begin{array}{c|cccccccccc}
        &-4&-3&-2&-1&0&1&2&3&4&\ldots \\\hline
        -4&0 & 1 & 3 & 8 & 21 & 13 &4 & 2 & 1 & \ldots \\
        -3& & 0 & 1& 3& 8& 5 & 2 & 1 & 1 & \ldots \\
        -2&&& 0&1& 3 &2 &1&1&2 & \ldots \\
        -1&&&&0&1& 1&1 &2 &5 &\ldots \\
        0&&&&&0& 1&2 &5&13& \ldots \\
        1&&&&&&0&1&3&8 &\ldots\\
        2&&&&&&&0&1&3&\ldots \\
        3&&&&&&&&0 & 1&\ldots
    \end{array}  \label{Fig:leapfrogfrieze}
\end{align}

One sees that the elements of $\S$ are ordered in the shape of a zigzag. Moreover, all entries of the frieze are Fibonacci numbers. 

Here indeed we find that $\S$ is equivalent to a triangulation giving the Penrose tiling with index sequence $(0, 0, \ldots)$: first we choose as fountain point $0$. Then the sequence $\xx^R=(0,0, \ldots )$ since there are no arcs of the form $(0,n)$ in the triangulation corresponding to $\S$. Choosing any other fountain point $i \in \ZZ$ will give us a sequence $\xx^R$ that is equivalent to $(0, 0, \ldots)$: if $i  \geq 0$, then the $\xx^R$-sequence is $(0,0, \ldots)$, if $i <0$, then after finitely many mutations (only one mutation at $i$ is sufficient), the triangulation also yields the $\xx^R$-sequence $(0,0, \ldots)$ and thus by Theorem \ref{Thm:fountains-equiv} the two fountains are equivalent.

\newcommand{\etalchar}[1]{$^{#1}$}

\end{document}